\pdfoutput=1
\documentclass[12pt]{article}
\usepackage{geometry} 
\usepackage{graphicx}
\usepackage{subfig}
\usepackage{float}
\usepackage{url}
\usepackage{mathtools}
\usepackage{amsthm} 
\usepackage{cite}

\theoremstyle{definition}
\newtheorem{defn}{Definition}[section]

\theoremstyle{plain}
\newtheorem{theorem}[defn]{Theorem}
\newtheorem{lemma}[defn]{Lemma}

\theoremstyle{remark}

\usepackage{multicol}
\usepackage{amssymb}

\newcommand{\singlespaced}{\renewcommand{\baselinestretch}{1}\normalfont}
\newcommand{\doublespaced}{\renewcommand{\baselinestretch}{1}\normalfont}
\newcommand{\realdoublespaced}{\renewcommand{\baselinestretch}{2}\normalfont}
\singlespaced
\doublespaced

\numberwithin{equation}{section}

\def\thetitle{NUMERICAL ANALYSIS OF TARGET ENUMERATION VIA EULER CHARACTERISTIC INTEGRALS}
\def\theauthor{Sam Gittleman Krupa}

\def\theyear{2013}
\begin{document}

\large\newlength{\oldparskip}\setlength\oldparskip{\parskip}\parskip=.1in
\thispagestyle{empty} \realdoublespaced

\begin{center}
\vspace*{\fill} \thetitle

\theauthor

A THESIS

in

Mathematics

\large Presented to the Faculties of the University of
Pennsylvania in Partial
 Fulfillment of the Requirements for the Degree of Master of
 Arts

\large
\theyear
\end{center}

\vspace{10mm}

\noindent\makebox[0in][l]{\rule[2ex]{3in}{.3mm}}  \hspace{-3.5mm} \singlespaced Radmila Sazdanovic \newline Supervisor of Thesis

\vspace{10mm}

\noindent\makebox[0in][l]{\rule[2ex]{3in}{.3mm}} \hspace{-3.2mm} \singlespaced Jonathan Block \newline Graduate Group Chairperson \vspace*{\fill}

\normalsize\parskip=\oldparskip

\doublespaced

\vspace{24mm}
\vspace*{\fill} \begin{center} {\it
To mom, dad, and Mariel}
\end{center}
\vspace*{\fill}

\section*{Acknowledgments}
This thesis was written almost entirely under the advisement of Professor Michael Robinson, now at American University. I cannot thank Professor Robinson enough for his tremendous help and encouragement, for introducing me to research, and for providing my research topic. Thanks also to his newborn son Theodore, who patiently and adorably joined us for advising sessions. Professor Robinson was until recently a member of the Applied Algebraic Topology team here at Penn. The team is run by Professor Robert Ghrist, and I would like to thank the rest of Professor Ghrist's team for providing feedback on the many iterations of this thesis I have presented at the group meetings over the years. I especially thank Professor Ghrist for providing support for my research.

Thanks further to Professor Radmila Sazdanovic, currently a member of Professor Ghrist's team, for supervising my thesis defense in Professor Robinson's absence, and for teaching me algebraic topology in my final semester here at Penn. 

Moreover, I would like to thank many of the other aspiring mathematics students at Penn. This includes Max Kieff (W'14) and Angelo Lee (W'13). They were always there if I wanted to talk a problem through or discuss a bug in one of my ideas. And special thanks to Robert Sharp (C'13), who found a few errors in earlier drafts of this work.

I submitted an earlier draft of this paper to \emph{Discrete \& Computational Geometry}. I would like to thank the referees for their feedback. Many of their ideas have been included in this paper in the further-work section (Section \ref{conclusion}).

And most importantly, thanks to my family for providing more support than I often realize. Without fail, they believe in my work, goals, and abilities, even when I don't. I couldn't be more thankful.

I gratefully acknowledge funding from ONR N000140810668.

\include{Abstract}
\begin{abstract}

Given a continuous sensor field, we can apply the Euler characteristic integral approach to count the number of targets  in the sensor field. If the sensor field is discrete, the Euler integral approach introduces errors into our target count. In this paper, we study the behavior of the Euler integral when applied to discrete sensor fields. Under precise assumptions, we count the number of first- and second-order errors in target count, and  discover a formula proportional to much higher order errors. This allows us to derive a point estimator for the number of targets in a discrete sensor field. Finally we derive an asymptotic result, providing insight into how the discrete Euler integral behaves for a large number of targets. 

   \end{abstract}

\newpage

\tableofcontents
\newpage

\pagenumbering{arabic}
\pagestyle{myheadings} \markright{}

\restylefloat{figure}
\relpenalty=9999
\binoppenalty=9999





\section{Introduction}
A \emph{sensor field} is a collection of sensors. Imagine a planar sensor field that is infinitely dense, with infinitesimally small sensors. At each point in our sensor field, there is a sensor which records simply the number of objects directly over it. Then, throw many objects of any shape onto this field. We require that the two-dimensional projection of each object be a compact connected set. We also require that all of the two-dimensional projections contain the same number of holes. We also require that the number of holes in the two-dimensional projections cannot be equal to one (so the annulus is not allowed). If this is the case, then according to the seminal work of   Yuliy Baryshnikov and Robert Ghrist \cite{ghrist09,baryshnikov10}, the number of objects (or \emph{targets}) can be exactly determined using nothing but the sensor field data recorded, i.e. the location of each sensor and how many objects overhead it detects. This concept of counting the number of holes in the two-dimensional projection is captured by the idea of the geometric  Euler characteristic \cite{ghrist09}.

If the sensor field is discrete and not of infinite sensor density the above method can be applied to this target enumeration problem, but it will only give an estimate. The most obvious error occurs when two targets do not overlap, but there is no sensor in between them to differentiate the two targets from one another. In this case, the two targets instead appear  to be one large target (see Figure \ref{fig1}). No algorithm could detect two targets here unless information on the targets' shapes is given.

\begin{figure}[h]
       \hspace*{130pt} \includegraphics[width=0.50\textwidth]{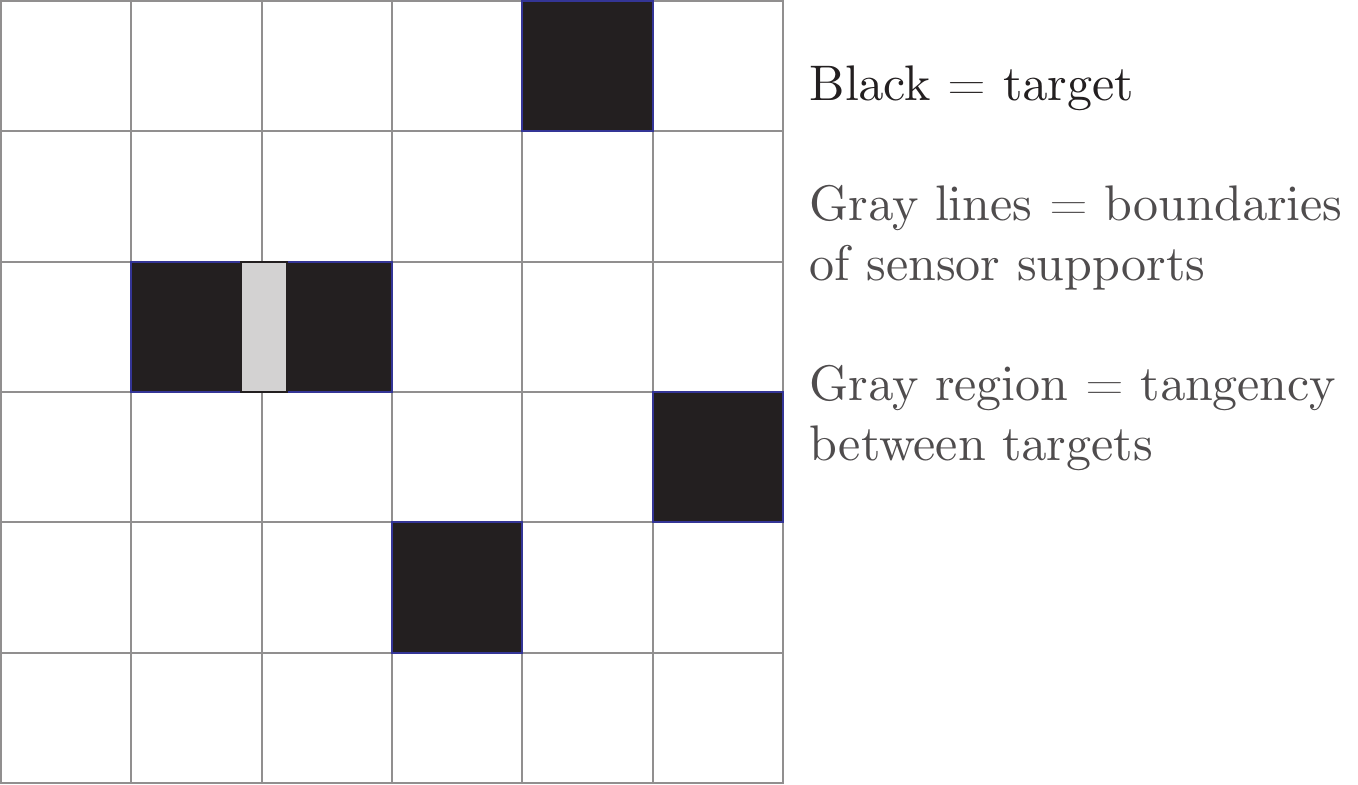}
  \caption{Example of a sensor grid. The union of each sensor's $1 \times 1$ area makes up the entire field. When two targets are almost tangent with no overlap but no sensor in between, they become impossible to differentiate. The tangency above (highlighted in gray) causes the two targets to appear as one. }
 \label{fig1}
\end{figure}

\begin{figure}
\centering
      \includegraphics[width=0.5\textwidth]{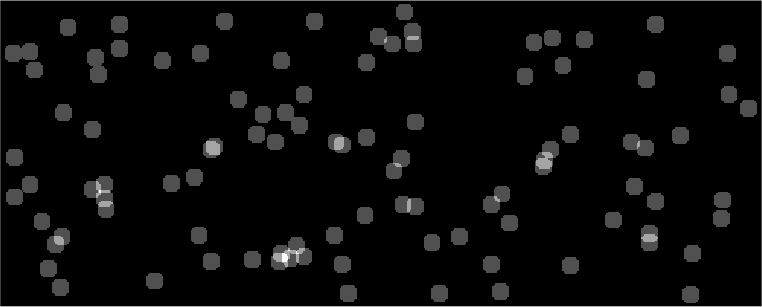}
  \caption{A visualization of the return of a typical discretized sensor grid, as will be studied throughout this paper. This sensor grid is $200\times500$, and 100 disk-shaped targets of radius 6 have been placed over it. A sensor with as many or more targets over it then all other sensors is painted white, a sensor with no target over it is painted black, and all other sensors are painted on a sliding scale between white and black.}
 \label{example_shot_100_targets_200by500_r6}
\end{figure}

This paper provides numerical analysis of the target enumeration via  Euler characteristic integral technique given by Baryshnikov and Ghrist  \cite{ghrist09,baryshnikov10}. In particular, we examine different sources of error when our sensor field is discretized and determine how to account for them. We begin by analyzing the causes of the most common first order (Section \ref{first_order_section}) and second order errors (Section \ref{second_order_section}) in target enumeration when the sensor field is discretized. The probabilities of these different errors appearing, under the assumption that the targets' two-dimensional projections are disk shaped and targets fall onto the sensor field under a discrete uniform distribution, are estimated using combinatorial arguments.  We then derive an asymptotic result showing how the Euler characteristic integral behaves for large numbers of targets (Theorem \ref{asymptotic_thm}). Finally, we find that our second order error formula is proportional to much higher order errors, and use this knowledge to approximate the number of targets resting in a sensor field (Section \ref{applications}).

\subsection{A survey of the field}
\label{survey}
Tomorrow's man will step into a world that is inundated with what we today call ``sensors.'' Fields of sensors will count him as he steps onto the street, as he drives his hover car, and when he steps onto the battlefield. Tiny sensors will precisely read his DNA. Other sensors will tirelessly count the pesticide molecules on his vegetables. Sensors of all sorts will appear to read his mind, returning data to him just when he needs it.  These same sensors will feed smart algorithms working to improve traffic flow, crop development, and battlefield awareness \cite{estrin02}. But before this world can exist for our everyman to experience, much work needs to be done.

The study of sensor networks is nascent, and the fundamentals are still being developed. There are many active research areas in this exciting new field. There are open problems related to networking, communication,  and the signal processing algorithms used to connect individual sensors into working, collaborative sensor fields using minimal computing resources and energy (\hspace*{-1.3mm}\cite{boulis03,fang03,guibas02,li02}, amongst many others).  Another important open problem is how to monitor and make sense of the large amounts of data coming off real world sensor fields with broken and malfunctioning sensors \cite{zhao03}.  And finally, complete mathematical theories for target enumeration and tracking via fields of sensors are just starting to be developed. 

The problem of interest here is \emph{target enumeration}: given the return of a sensor field, how best to fuse the data from the individual sensors so as to count the total number of unique targets being picked up by the sensor field.

Much of the current work in target enumeration or the closely related field of target tracking \cite{guibas02,he06,jung02} assumes targets are able to be easily differentiated from one another. However, a major open problem is the \emph{complete} target counting and tracking problem where targets are less easily differentiated from one another and from their environment. Another application of this work might be instances where we are not allowed to use identification because of privacy concerns.

Wo provide two specific examples of work that has been done. One example involves sensors which return discrete target counts, and the other example involves sensors which return real-valued target counts. The first example involves binary sensors, a type of discrete sensor. It has been found that for any fixed time, not enough information is contained in a binary sensor field to return an accurate estimate of the number of targets \cite{singh07}. But by looking at time series data through a particle filter, both target counts and target trajectories can be nicely approximated under the weak assumption of relatively small changes in the target velocity vector moment to moment. However, the techniques have so far only been developed in one space dimension and no theoretical efficacy analysis has been done.

In contrast to discrete sensors are sensors which return real-valued target counts. An example of this kind of sensor field is one where each sensor picks up a voltage related to the number of targets around that sensor, and the distances between the targets and the sensor. With this type of sensor field, it is possible to count the number of targets sensed by counting the number of relative maxima in the voltage data, when this data is thought of as a graph. However, current methods fail when targets are insufficiently separated or are in other complex arrangements \cite{fang03}. The authors in \cite{fang03}  try to reduce errors by intelligently turning on and off sensors and carefully tracking the effects of noise. They place approximate bounds on the expected errors in their methods, but again no close analysis of how errors accumulate under their algorithms has been performed.

The seminal work to date on a solution to the problem of target enumeration has been done by Y. Baryshnikov and R. Ghrist using the topological technique of the integral with respect to the Euler characteristic \cite{ghrist09,baryshnikov10}. In the case of infinitely dense continuous sensor fields, their breakthrough new theoretical framework  is able to accurately count targets which are in very complex configurations using sensors which return only local target counts. Further, the expected Euler characteristic integral of a Gaussian random field has  been worked out, which allows for unbiased target enumeration from noisy sensor fields \cite{bobrowski10}.  However, until now nothing more than conjecture has existed concerning numerical analysis of the Euler characteristic integral when our sensor field is discretized.

We begin here with a preliminary numerical analysis of the Euler characteristic integral as a tool for target enumeration, picking up where Baryshnikov and Ghrist left off. We make precise assumptions. Then, we derive expectations for the number of target counting errors in a discrete sensor field when the enumeration via Euler characteristic integral theory is applied. We gain an understanding into how different types of discretization errors occur in target counting. The result is the only technique known to the author that gives an unbiased approximation to the number of targets in a discrete sensor field returning only local target counts at each sensor.

\subsection{Background}
We present the main result we use in topological target enumeration. First, we define some terminology. We think of our sensor field as a subset $\Omega$ of $\mathbb{R}^2$. The targets can be thought of as being not point particles but two-dimensional laminae (subsets of the plane). A sensor field, discrete or continuous, has a corresponding \emph{height function} $h(x) \in \mathbb{Z}$ which at each point $x$ in the sensor field $\Omega$ returns the number of targets the sensor at $x$ is detecting. Note that even when the sensor field is discrete, its height function can still be thought of as coming from a continuous sensor field. 
\begin{defn}[Excursion sets \cite{ghrist09}] For an integer $s$, a topological space $X$ and for a function $h: X \to \mathbb{Z}$, the set $\{h=s\}$ is the \emph{level set} $\{x \in X: h(x) =s\}$, and the set $\{h > s\}$ is the \emph{upper excursion set} $\{x \in X: h(x) >s\}$. Lower excursion sets are likewise defined. 
\end{defn}
Assume  each target, thought of as a two-dimensional laminae, has Euler characteristic equal to one. Then when the sensor field is continuous, the integral of its height function with respect to Euler characteristic returns the number of targets in the sensor field. We calculate the Euler characteristic integral using the duality formula derived by Ghrist and Baryshnikov:
\begin{theorem}[{\cite[Theorem 4.3]{ghrist09}}]
\label{theoretical_target_enumeration_theorem}
For $h:  \mathbb{R}^2 \to \mathbb{N}$ constructible and upper semicontinuous,
\begin{align}
\int_{ \mathbb{R}^2}h\,d\chi=\sum_{s=0}^{\infty}(\beta_0\{h>s\} -\beta_0\{h\leq s\}+1), \label{duality}
\end{align}
where $\beta_0$ denotes the zeroth \emph{Betti number}, or equivalently, the number of connected components of the set.
\end{theorem}
\emph{Constructible} functions are ``tame" and integer-valued on a topological space (for a quick introduction, see \cite{baryshnikov10}). 

While not  tested in this article, we expect  \eqref{duality} to be the most numerically stable of all equivalent formulae \cite{RobinsonSaid}. 

\section{Numerical Analysis}
\label{numerical_analysis}
\label{conditions}
In this paper we study ideal discrete sensor fields, which are made up of a finite number of sensors, all of which are assumed to be performing their function without error. The object of our study are discrete sensor fields which are made up of collections of square sensors of dimension $1\times 1$ which detect the number of targets above them without error. These sensors can also be thought of as point particles with square sensing regions. The sensors are then arranged in regular grids with no overlap or space among the sensing regions.

In this paper we will think of our sensor field as a subset $\Omega$ of $\mathbb{R}^2$, and we will work with the following height function. Define a height function $h$ on $\Omega$ such that for $x$ in $\Omega$ not along the boundary of any sensor's sensing region, $h(x)$ returns the total number of targets above that sensor whose sensing region contains $x$. In our construction of an ideal discrete sensor field, there is no sensing region along the boundaries of the sensor supports, and any target without width that lies entirely along the boundaries of sensor supports will not be picked up by our sensor field. Then for $x$ in $\Omega$ along the boundary of a sensor support, we define $h(x)$ such that its value agrees with the value returned by the sensor underneath or to the left of $x$. 

This discretization of the sensor field changes the basic behavior of the Euler characteristic integral in two important ways, which we discuss in Section \ref{finite_discete_errors}.

\subsection{Finite field discretization errors}
\label{finite_discete_errors}
\vspace{-1mm}
\begin{figure}
  \centering
  \subfloat[Without  zero-padding, \newline $\int_{\Omega}h\,d\chi=20$]{\label{no_zero_padding}\includegraphics[width=0.35\textwidth]{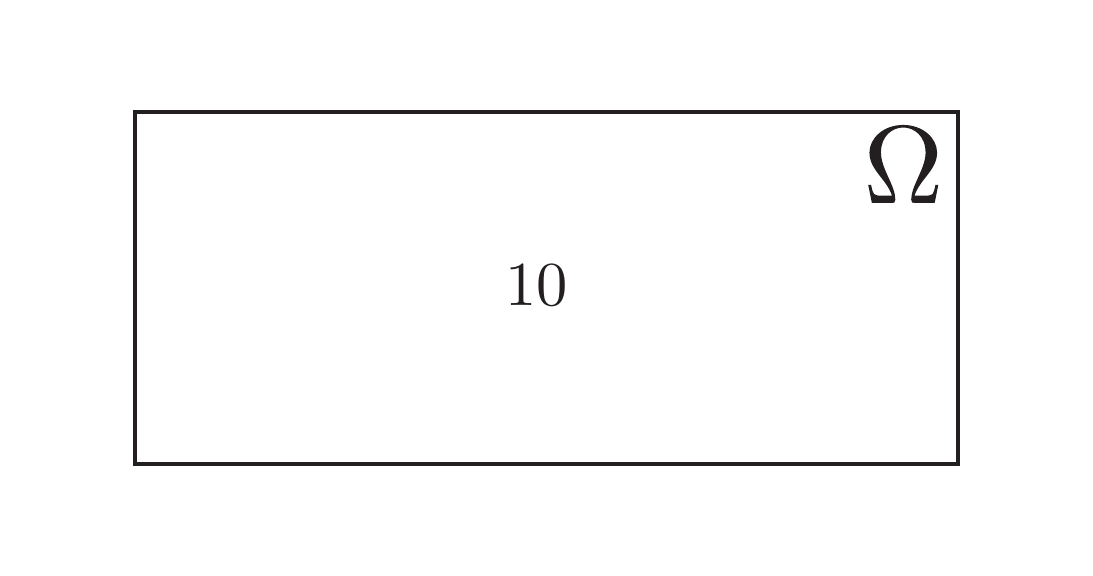}}  
  \hspace*{25mm}              
  \subfloat[With zero padding, \newline$\int_{\Omega}h\,d\chi=10$]{\label{zero_padding}\includegraphics[width=0.35\textwidth]{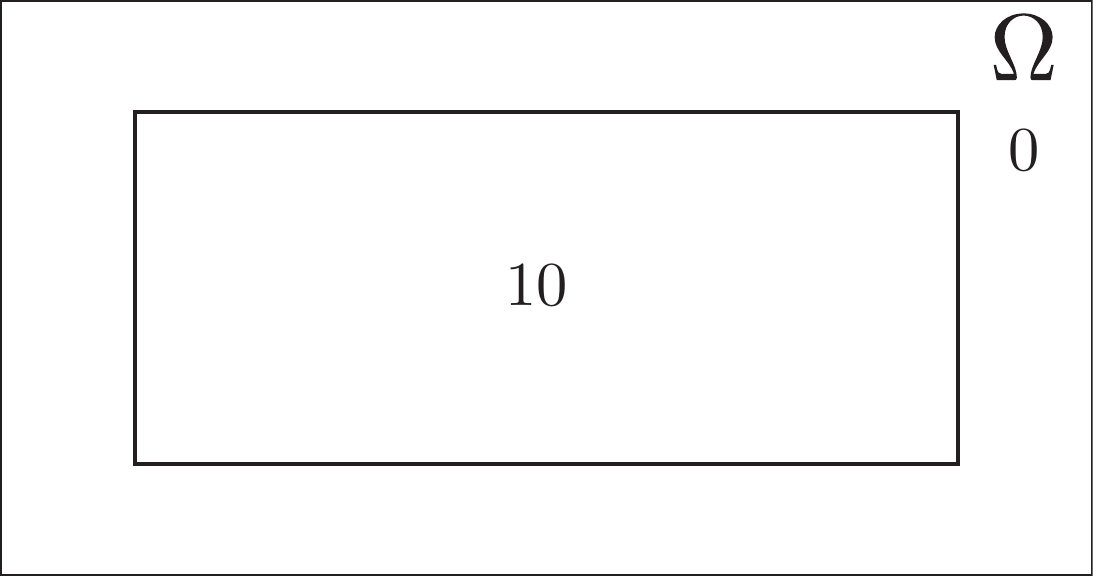}}                
                        
  \caption{The Euler characteristic integrals of the returns of these two sensor fields show the effect on the integral caused by  zero-padding the sensor field data.}
  \label{zero_padding_comparison}
\end{figure}

When using Theorem \ref{theoretical_target_enumeration_theorem} to evaluate the Euler characteristic integral over a \emph{bounded} and connected region $\Omega \subset \mathbb{R}^2 $, we need to be careful because the assumptions are not satisfied ($\Omega \neq \mathbb{R}^2 $). Suppose, for instance, the height function $h=10$ throughout $\Omega$, so that $h$ never vanishes along $\partial\Omega$. Then the calculated value of the integral using Theorem \ref{theoretical_target_enumeration_theorem} is 20.  This differs from the correct value of 10 by a factor of 2 (see Figure \ref{zero_padding_comparison}). To correct this error, we can zero-pad $h$ by constructing a function $g$ such that,

\[
 g(x) =
  \begin{cases}
   h(x) & \text{if } x \in \Omega \\
   0       & \text{if } x \not\in \Omega
  \end{cases}
\]
and then integrating $g$ over $\mathbb{R}^2$ instead of $h$ over $\Omega$. For the rest of this paper, this correction is always assumed to be made when working with finite, discrete sensor fields.

\subsection{Linearity of the discrete Euler characteristic integral}
Under discrete sensor fields, linearity of the Euler characteristic integral  (which is true \cite{eulertome} for continuous sensor fields) does not hold.  Take for instance  two height functions ($h_1$ and $h_2$) each generated by sensor fields of four sensors arranged in a square grid, where one small target is placed on each field. Then, with black indicating function value of $1$ and white indicating $0$,
\begin{center}
 \includegraphics[width=0.8\textwidth]{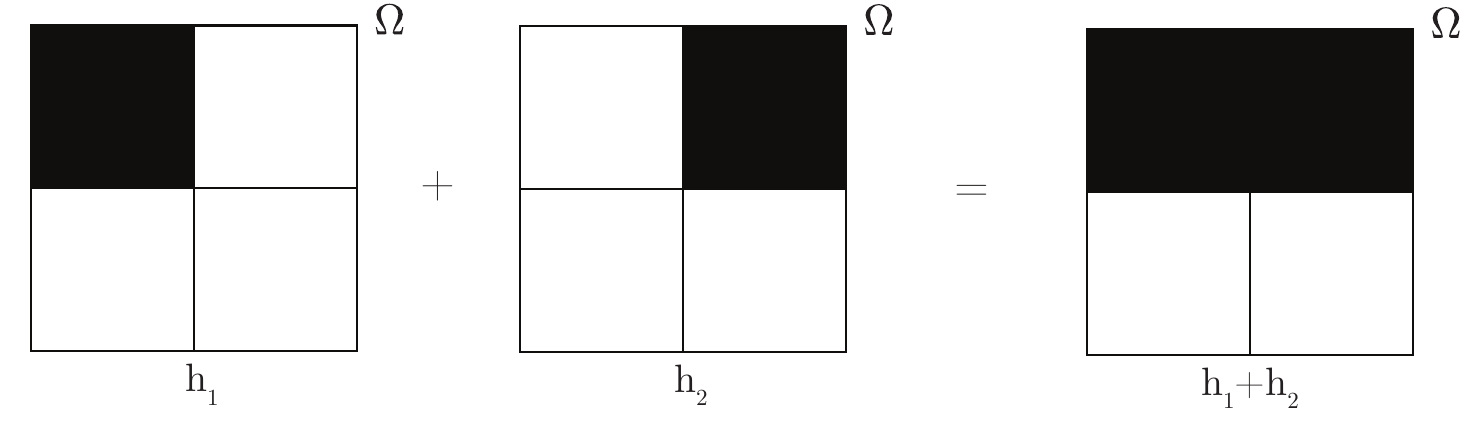}.
\end{center}
Note that $\int_{\Omega}h_1\,d\chi=1$ and $\int_{\Omega}h_2\,d\chi=1$, but $\int_{\Omega}(h_1+h_2)\,d\chi=1\neq1+1=2$. Thus, over discrete sensor fields the Euler characteristic integral is not a linear operator.

However, one particular instance of linearity still holds for discrete sensor fields.
\begin{lemma}
For any positive integer $m$,
\begin{align}
\int_{\Omega}\sum_{i=1}^m h\,d\chi &= m\int_{\Omega}h\,d\chi \label{linearity_relation}
\end{align}
\end{lemma}
\begin{proof}
This is a result known \cite{eulertome} to hold over continuous sensor fields, but as seen above that does not necessarily mean the result holds over discrete sensor fields. Thus we argue as follows to show the result continues to hold over discrete sensor fields.

Note,
\begin{align}
\beta_0\{mh>mk+i\}&= \beta_0\{h>k\}  \label{level_relation1} \\
\beta_0\{mh\leq mk+i\}&= \beta_0\{h\leq k\}  \label{level_relation2}
\end{align}
for all $i=0,1,2,\dots, m-1$ and for all $k=0,1,2,3\dots$.
Then,
\begin{align*}
\int_{\Omega}\sum_{i=1}^m h\,d\chi &=\sum_{s=0}^{\infty}(\beta_0\{mh>s\} -\beta_0\{mh\leq s\}+1) \\
&=\sum_{i=0}^{m-1}\sum_{k=0}^{\infty}(\beta_0\{mh>mk+i\} -\beta_0\{mh\leq mk+i\}+1) \\
\shortintertext{and then by \eqref{level_relation1} and \eqref{level_relation2},} 
&=\sum_{i=0}^{m-1}\sum_{k=0}^{\infty}(\beta_0\{h>k\} -\beta_0\{h\leq k\}+1) \\
&=m\sum_{k=0}^{\infty}(\beta_0\{h>k\} -\beta_0\{h\leq k\}+1) \\
&=m\int_{\Omega}h\,d\chi
\end{align*}
which proves the statement.
\end{proof}
\subsection{Analysis of target enumeration}
The  numerical properties of Theorem \ref{theoretical_target_enumeration_theorem}, as a point estimator for target enumeration in discrete sensor fields, depend entirely on the magnitude, shape, and distribution of the targets, as well as the setup of the sensor field. For example, the smaller the targets, the less likely a tangency is to occur for a given sensor field size. We first consider tiny targets, smaller than the individual sensor size. Given a large sensor field, the probability that two of these tiny targets will appear tangent and cause an error is much smaller than when large disks are being placed on a sensor field of the same size, as each large disk has a large perimeter to which another target can come in contact. Noncircular objects have their own sets of problems. Imagine, for instance, a star-shaped target and a disk-shaped target of similar magnitude: the disk has far less perimeter, and cannot interact with other disks in the way a star could interact with other stars. Moreover, the more uniformly spread out our targets, the less likely an error is to occur, in contrast  to all the targets being concentrated in a single part of our sensor region $\Omega$. So, for our analysis we will work under a few assumptions:
\begin{enumerate}
\item
The sensor field is  a rectangular region, of height $x$ and width $y$. The sensor field is made up of a regular grid of $x y$ sensors, each having a size of $1 \times 1$. The union of the sensors gives us the complete sensor field.
\item 
There are $n$ targets, each of which are open disks of the \emph{same} integer radius $r\leq7$. We avoid the use of closed disks because they produce extraneous features when they are discretized, and these features introduce new errors into the Euler characteristic integral (see Figure \ref{equation_comparison}).
\item
Targets are discretely placed with integer center coordinates.
\item
Targets' centers are uniformly distributed over our region.
\item
If a disk doesn't land entirely in the sensor field, this is called an \emph{edge effect}. To prevent edge effects, targets' centers are kept a minimum of $r$~units away from the edge.  
\end{enumerate}

For a detailed analysis of discrete integer radius circles with integer centers (known as \emph{Freeman digitizations}) see \cite{kulpa79}. 

\begin{figure}
  \centering
  \subfloat[\scriptsize{Discretized closed~disk.}]{\label{disk_with_artifacts}\includegraphics[width=0.25\textwidth]{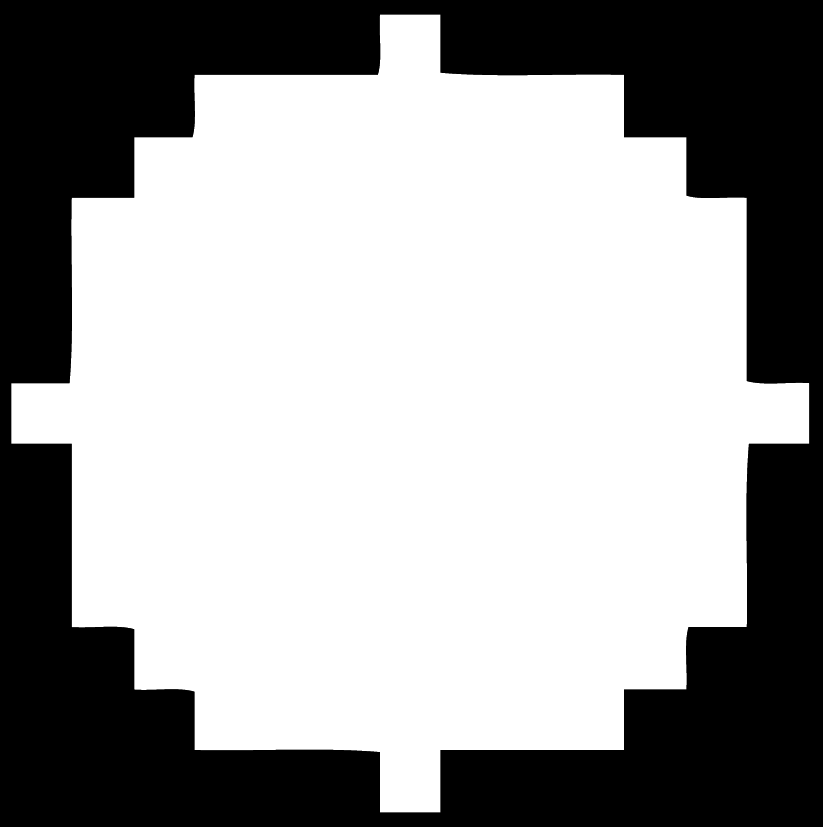}}
    \hspace*{25mm}                              
  \subfloat[\scriptsize{Discretized open~disk.}]{\label{disk_no_artifacts}\includegraphics[width=0.25\textwidth]{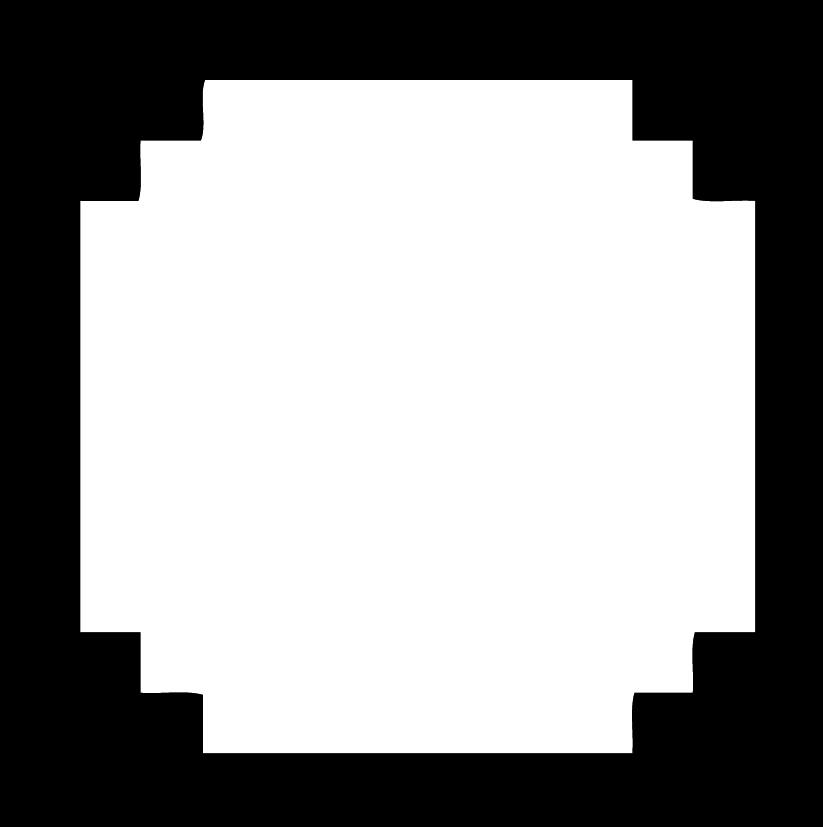}}                
                        
  \caption{A comparison of the difference between discretizing open and closed disks.}
  \label{equation_comparison}
\end{figure}

We proceed with calculating the expectation of $\int_{ \mathbb{R}^2}h\,d\chi$ when the integral is calculated using Theorem \ref{theoretical_target_enumeration_theorem}, and the assumptions are as above. Throughout this paper, we define $\mathrm{E}(\cdot)$ as the expectation function, and $\Pr(\cdot)$ as the probability function. Define \emph{errors} to be the (signed) difference between $n$ and the number of targets calculated using the Euler characteristic integral approach, where if $n$ is larger than the calculated number of targets, then there are positive errors. We let $\mathrm{E}_n$ denote the expected number of errors in our target count when using the Euler characteristic integral approach.

\subsection{First order approximation}
\label{first_order_section}

By ``first order errors,'' we mean, roughly, errors occurring with probability proportional to the probability that any two particular disks are tangent.

The goal in this section is to calculate the expected number of errors, to a first order approximation, when using the Euler characteristic integral approach to estimate the number of targets in a discrete sensor field. Our assumptions are as above:  the sensor field is a rectangular region of height $x$ and width $y$. We place $n$ disk-shaped targets down on our sensor field. Recall that $\mathrm{E}_n$ denotes the expected value for the numbers of errors. A first order approximation of $\mathrm{E}_n$ would be to count the expected number of tangencies amongst the targets, and take this as an estimate of the error. Let $T_{ij}$ be the indicator function,

\[
 T_{ij} =
  \begin{cases}
   1 & \text{if } \mbox{disk i is tangent to disk j }\\
   0       & \text{if } \mbox{disk i is not tangent to disk j}
  \end{cases}
\]

defined for all $i,j=1,2,\dots, n$ with $ i\neq j$. Then $\mathrm{E}\left(T_{ij}\right)=\Pr(T_{ij}=1)$. As the the locations of the disks' centers are independently and identically distributed, all tangencies occur with the same probability. 

For a fixed radius $r$, and a fixed sensor field region larger than $4r\times 4r$
\begin{equation}
\label{tangency_probability}
\Pr(T_{ij}=1)=\frac{{C_{r}}}{(x-2r+2)(y-2r+2)},
\end{equation}
for all $i,j=1,2,\dots, n$ with $ i\neq j$, and where ${C_{r}}$ is the average number of positions around one disk where a tangent disk could be placed.. This follows because the disks' centers are uniform random variables, and only an area of $(x-2r+2)(y-2r+2)$ is available for the centers to be placed. 

By the assumptions above we are preventing edge effects, so in calculating ${C_{r}}$, we need to correct for the fact that disks closer to the edge cannot have another disk placed to one side of them (see Figure \ref{fig3}). The magnitude of this correction is most notable when placing radius $r$ disks down on a $2r\times2r$ sensor field without edge effects: no matter now many disks are put down, as there is only one location where the disks can fall, no tangencies will occur.

\begin{figure}
\centering
      \includegraphics[width=0.7\textwidth]{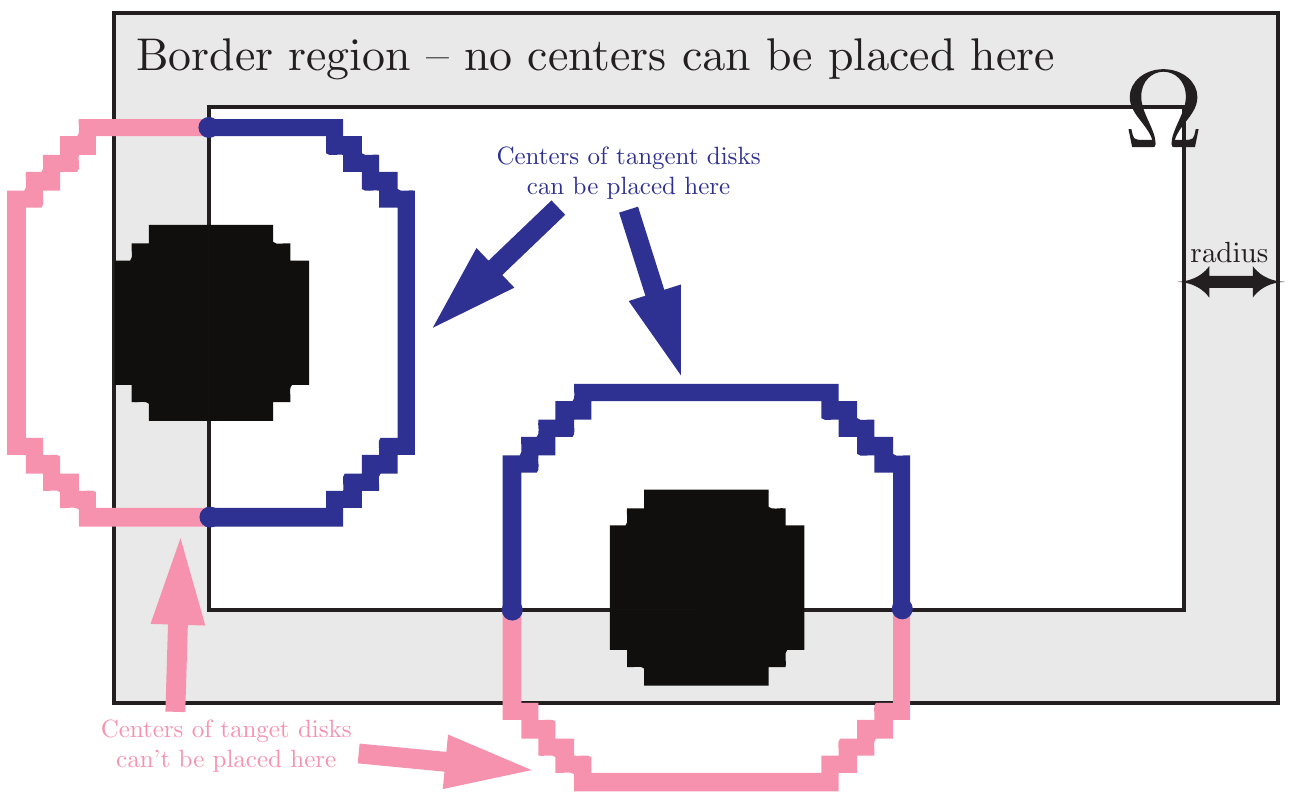}
  \caption{A disk with a center further than $3\cdot r$ from the edge of the sensor region can have a tangent disk placed at many spots around the disk. However, a disk with a center closer than $3\cdot r$ from the edge of the sensor region will have fewer positions where a tangent disk can be placed. Each disk in the sensor field pictured here has a ring around it. Each point in the ring is a point where a tangent disk could be centered. Note, however, that not all of the points in a ring are within the region where  centers can be placed.}
 \label{fig3}
\end{figure}

We calculate ${C_{r}}$. Assume our region is larger than $4r\times 4r$. At each of the points between $r$ and $3\cdot r$ from the edge, center a disk. For each disk calculate at how many points in our region a tangent disk could be placed (call these \emph{tangency points}). Then calculate the average number of tangency points for a disk centered at one of these locations close to the edge. Then we can calculate a weighted average representing the number of tangency points the average disk in our region is expected to have (assuming disks are placed down under a discrete uniform distribution). The number of tangency points around the average discrete disk in our sensing region will be lower then the number of tangency points around a discrete disk which is further than $3\cdot r$ from any edges.

Then, for a sensor field $\Omega$ of height $x$ and width $y$,
\begin{equation*}
 {C_{r}}= \frac{(x\cdot y - (y-4\cdot r)\cdot (x-4\cdot r))\cdot \alpha+(y-4\cdot r)\cdot (x-4\cdot r)\cdot \beta}{A(\Omega)}
\end{equation*}

where

\begin{itemize}
\item $\alpha $ denotes the average number of tangency points for a disk with center between $r$ and $3\cdot r$ from the edge
\item $\beta$ denotes the number of tangency points for a disk with center further than $3\cdot r$ from the edge
\item $A(\Omega) $ denotes the total area of the sensing region $\Omega$ 
\end{itemize}
With this we can correct for edge effects, but corner effects are left unchecked (a disk in a corner has even fewer tangency points then a disk simply near an edge). However, for reasonably large sensor grids, corner effects will have negligible effect.

 For example, in the case of disks of radius 6 there are 88 points around a disk with a center further than $18$ units from the edge  where the center of a second identical disk can be such that the second disk causes a tangency with the first disk. This is shown by the white ring around the discrete disk of radius 6 in Figure \ref{fig2}.

\begin{figure}[tb]
    \hspace*{120pt}   \includegraphics[width=0.5\textwidth]{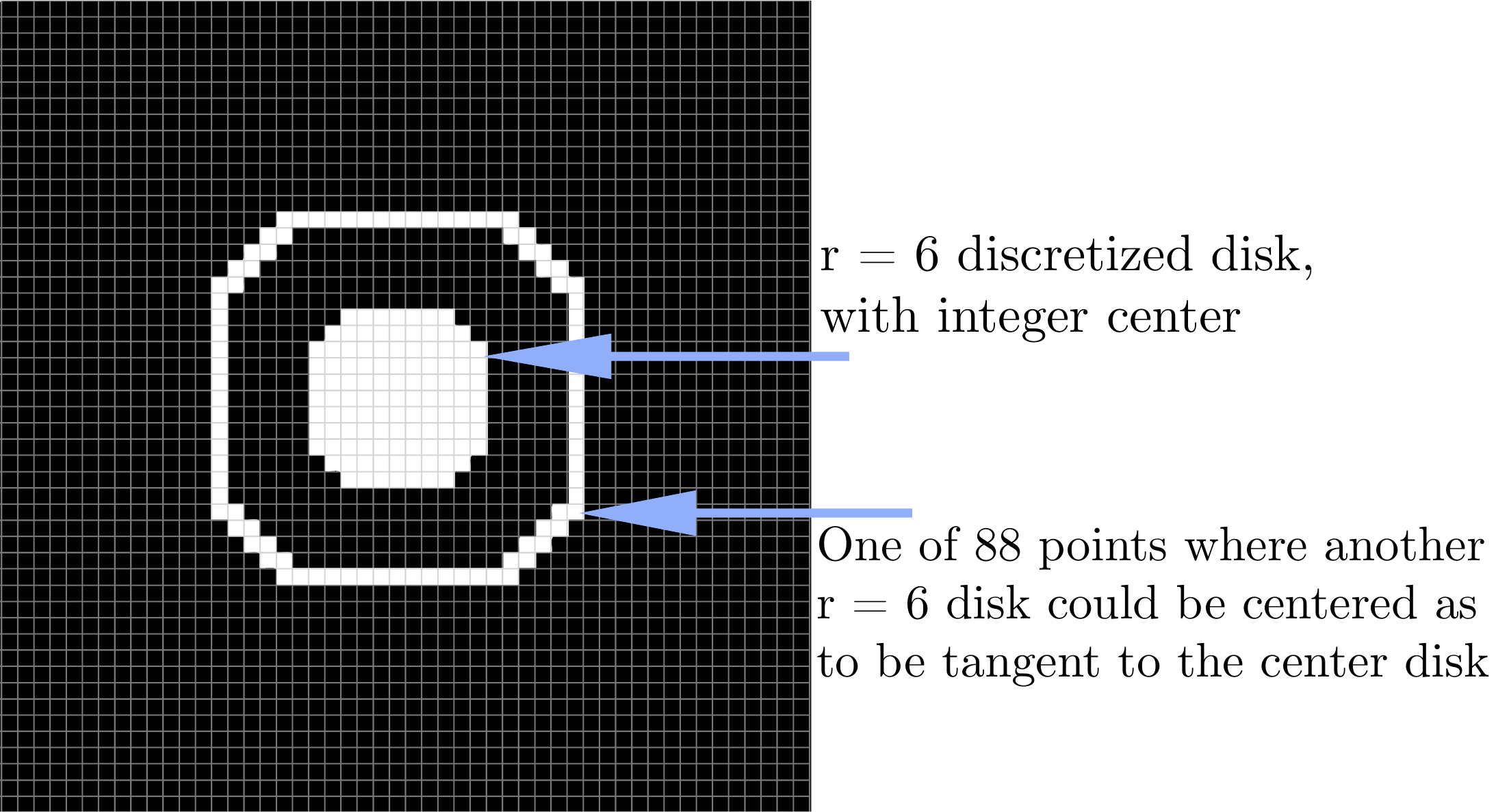}
  \caption{For disks of radius 6, the white ring contains 88 points where a disk's center can be placed, such that that disk causes a tangency with the center disk.}
 \label{fig2}
\end{figure}

Then, the expected number of tangencies amongst the $n$ disks is 
\begin{equation*}
\mathrm{E}\left(\frac{1}{2}\displaystyle \sum_{i=1}^{n}\sum_{\substack{j=1\\i\neq j}}^{n} T_{ij} \right) 
\end{equation*}
where we need to divide by two to prevent double counting because if disk $i$ is tangent to disk $j$, then disk $j$ is tangent to disk $i$. Then by linearity of expectation,
\begin{align}
\mathrm{E}\left(\frac{1}{2}\displaystyle \sum_{i=1}^{n}\sum_{\substack{j=1\\i\neq j}}^{n} T_{ij} \right) &=\frac{1}{2}\displaystyle \sum_{i=1}^{n} \displaystyle \sum_{\substack{j=1\\i\neq j}}^{n} \mathrm{E}\left(T_{ij}\right) \nonumber
\shortintertext{Then, because the disks' centers are independently and identically distributed random variables, we can simplify this expression to get}
&=\frac{1}{2}n\cdot(n-1)\Pr(T_{ij}=1)\nonumber
\shortintertext{for any $i\neq j$. Substituting in \eqref{tangency_probability}, we find}
\mathrm{E}\left(\frac{1}{2}\displaystyle \sum_{i=1}^{n}\sum_{\substack{j=1\\i\neq j}}^{n} T_{ij} \right)&=\frac{n\cdot(n-1)}{2}\frac{{C_{r}}}{(x-2r+2)(y-2r+2)} \label{tangency_formula}
\end{align}

Then, recalling that a first order approximation of $\mathrm{E}_n$ would be to take the number of tangencies \eqref{tangency_formula} as an estimate for the number of errors,
\begin{align*}
\mathrm{E}_n&\approx\frac{n\cdot(n-1)}{2}\frac{{C_{r}}}{(x-2r+2)(y-2r+2)} .
\end{align*}

Experimental results (Table \ref{summary:formulas}) suggest equation \eqref{tangency_formula} holds as an approximation to the number of errors as long as targets (disks in this case) are not too densely packed allowing for many higher order errors to occur; but this approximation still becomes poor very fast. 

So we need to find the next set of second order errors. 

\subsection{Second order approximation}
Recall that the probability that two particular disks $i$ and $j$ are tangent is ${\Pr(T_{ij}=1)}$. Similar to first order errors, by ``second order errors,'' we mean, roughly, errors occurring with probability proportional to ${\Pr(T_{ij}=1)}^2$.

\label{second_order_section}
\begin{theorem}[Second order error approximation formula]
\label{second_order_theorem}
For simplicity we assume a disk radius of 6. As above, we let ${C_{6}}$ denote the average number of tangency points for a radius 6 disk. We also let $\hat{A}=(x-2r+2)(y-2r+2)$ be the area of the region where disk's centers can be placed. Recall that $\mathrm{E}_n$ denotes the  expected value for the numbers of errors in target enumeration when using the Euler characteristic integral approach.
Then, a second order estimate of $\mathrm{E}_n$ is given by 
\begin{equation}
\frac{{C_{6}}}{c^{2}}\left(\left(1-\frac{c}{\hat{A}}\right)^{n}\cdot \hat{A}-\hat{A}+nc\right)
\label{final_second_order}
\end{equation}
where $c$ is a proportionality constant and when the conditions assumed in the beginning of  Section \ref{conditions} are met.
\end{theorem}

As a first step towards proving Theorem \ref{second_order_theorem}, let's calculate the Euler characteristic integral of a few different sensor fields to get a feel for how errors accumulate. See Figure \ref{fig4}.

\begin{figure}
  \centering
  \subfloat[2~targets, $\int_{\Omega}h\,d\chi=1$]{\label{sensor_field1}\includegraphics[width=0.31\textwidth]{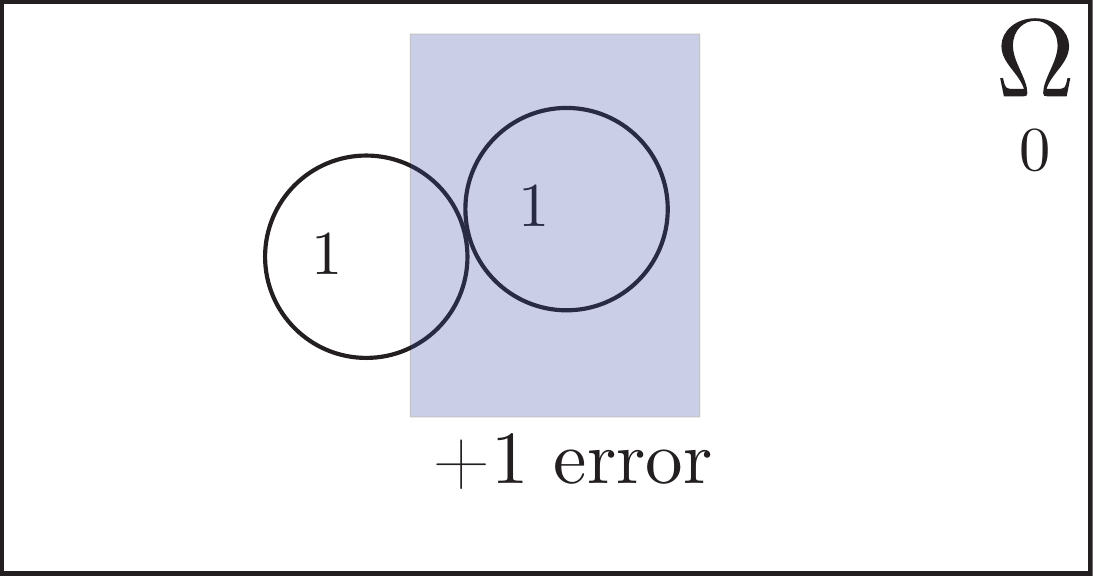}}\hspace*{3mm}                              
  \subfloat[3~targets, $\int_{\Omega}h\,d\chi=2$]{\label{sensor_field2}\includegraphics[width=0.31\textwidth]{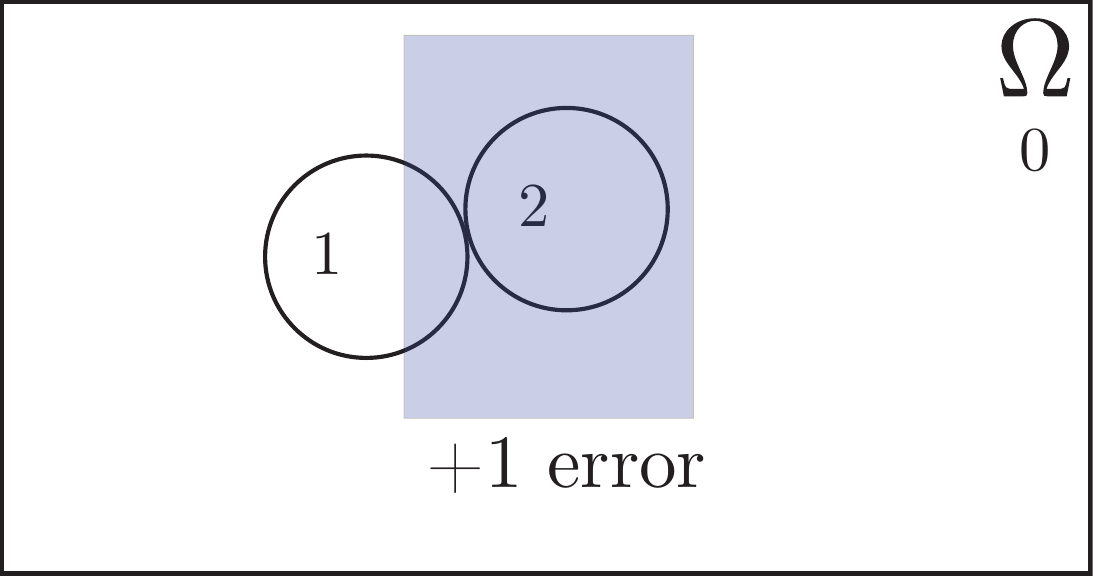}}\hspace*{3mm}                     
  \subfloat[3~targets, $\int_{\Omega}h\,d\chi=2$]{\label{sensor_field3}\includegraphics[width=0.31\textwidth]{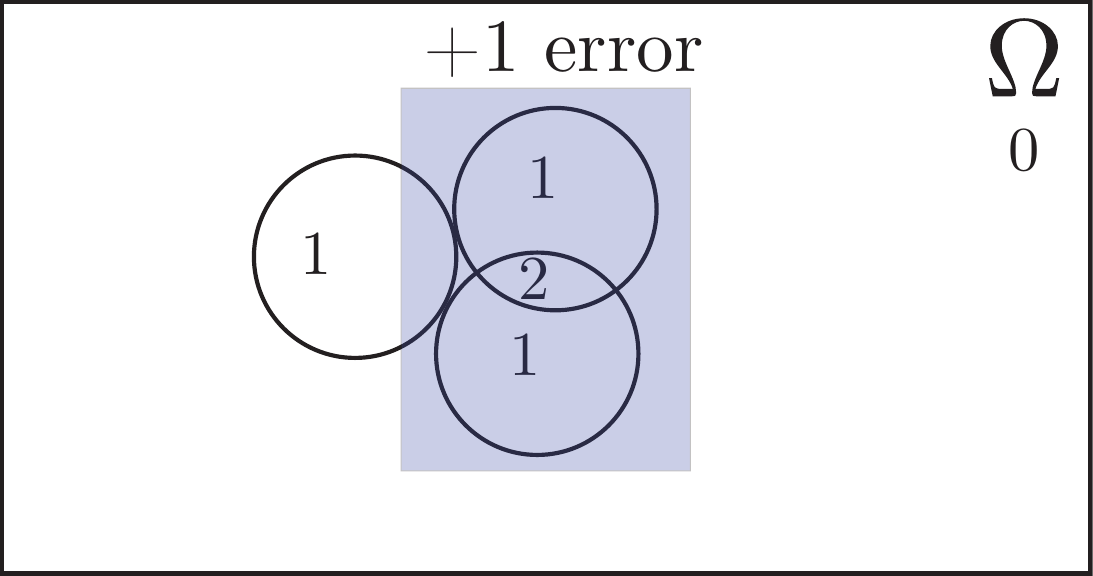}}  \\
  \subfloat[4~targets, $\int_{\Omega}h\,d\chi=3$]{\label{sensor_field4}\includegraphics[width=0.31\textwidth]{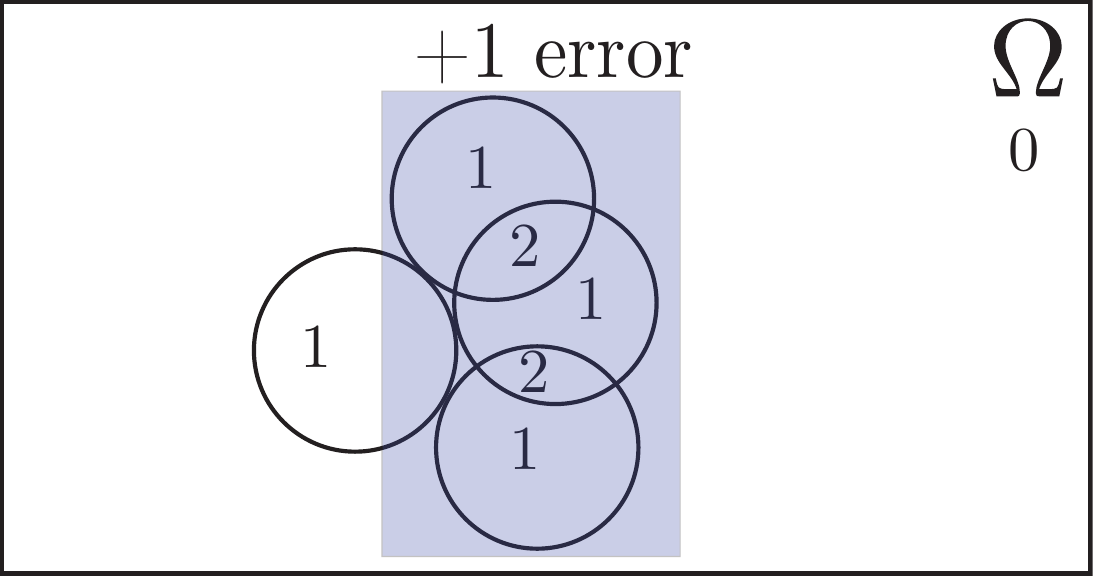}}\hspace*{3mm}                     
  \subfloat[5~targets, $\int_{\Omega}h\,d\chi=3$]{\label{sensor_field5}\includegraphics[width=0.31\textwidth]{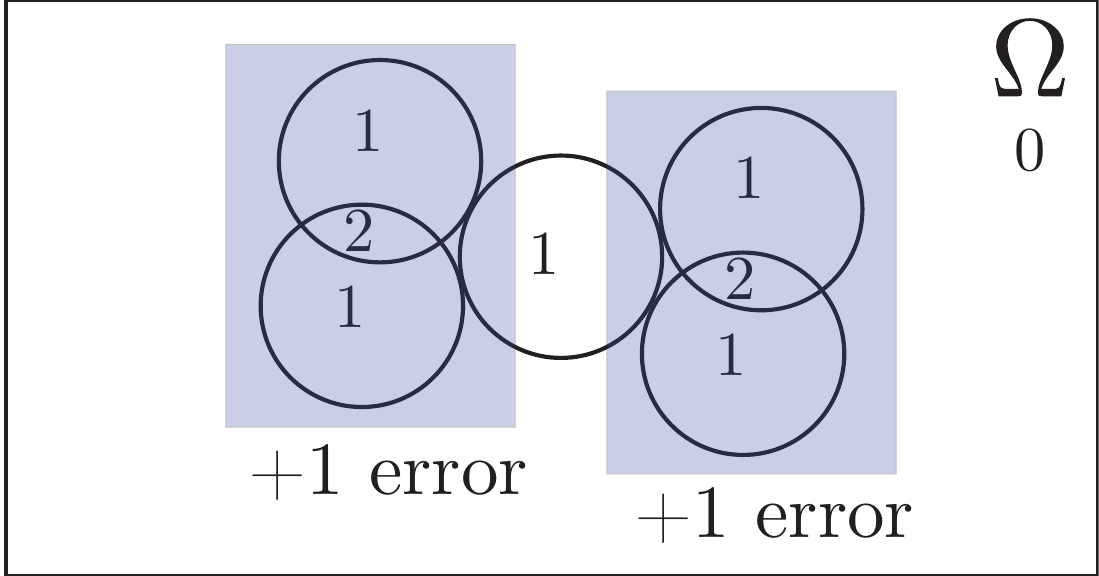}}                
              
  \caption{Above are the height functions returned by 5 different sensor fields, and their Euler characteristic integrals.}
  \label{fig4}
\end{figure}

In \ref{sensor_field1}, there is one tangency and one error. In \ref{sensor_field2}, however, there are two tangencies but still only one error. The third disk put down over the second \emph{discriminates} itself by creating an additional connected component in the $\{h=2\}$ level set. This is the key insight in our derivation of Theorem \ref{second_order_theorem}: going from \ref{sensor_field1} to \ref{sensor_field2}, the total number of errors does not increase. Appropriate generalizations also apply. We can shift the 3rd disk (see  \ref{sensor_field3}), and as long as it still discriminates itself against the 2nd disk the 3rd disk's tangency with the first disk does not lead to an error. Note similar phenomena in  \ref{sensor_field4} and  \ref{sensor_field5}.

On our way to proving Theorem \ref{second_order_theorem}, the above first suggests  the following lemma. 
\begin{lemma}
Assume we are given a discrete sensor field. Then $n$ disk-shaped targets of radius 6 are discretely and uniformly placed over it as assumed in the beginning of Section \ref{conditions}. Then,
\begin{equation}
\label{second_order1}
\mathrm{E}_n\approx\mathrm{E}_{n-1}+\frac{(n-1){C_{6}}}{\hat{A}}-\frac{c\cdot\mathrm{E}_{n-1}}{\hat{A}}
\end{equation}
\end{lemma}
\begin{proof}
Put down the first target,
\[\mathrm{E}_1=0\]
Put down the 2nd target, and there is one target it can come tangent to,
\[\mathrm{E}_2=\frac{{C_{6}}}{\hat{A}}\]
Put down the 3rd target and there are two targets it can come tangent to in such a way as to create an error, remembering that some tangent positions will also cause the 3rd target to be discriminated. The number of such positions is related to the number of tangencies already present, the expected number of which we already calculated above. So,
\[\mathrm{E}_3\approx\frac{{C_{6}}}{\hat{A}}+\frac{2\cdot{C_{6}}}{\hat{A}}-c\cdot\frac{(\mbox{previous number of tangencies})}{\hat{A}},\]
by the uniform distribution of target centers, and where c is a proportionality constant. We are subtracting off from the probability that the 3rd target creates a tangency, the probability that the 3rd target does not create an error, for as noted going from \ref{sensor_field1} to \ref{sensor_field2} the number of errors does not increase.
\begin{eqnarray*}
\mathrm{E}_3&\approx&\frac{{C_{6}}}{\hat{A}}+\frac{2\cdot{C_{6}}}{\hat{A}}-c\left(\frac{{C_{6}}}{\hat{A}^2}\right) \\
&=&\frac{3\cdot{C_{6}}}{\hat{A}}-c\left(\frac{{C_{6}}}{\hat{A}^2}\right)
\end{eqnarray*}

For radii of 6, and given two tangent disks, there are on average (for the 88 possible setups, with a ``center disk'' held fixed) 21.455 points where a case such as either \ref{sensor_field2} or \ref{sensor_field3} can occur, so we take $c=21.455$ for second order errors (see Figure \ref{calculate:c2}). The value 21.455 is an exact value calculated by an exhaustive search of all possible points for a third disk, averaged over all 88 possible setups.
\begin{figure}
\centering
      \includegraphics[width=0.5\textwidth]{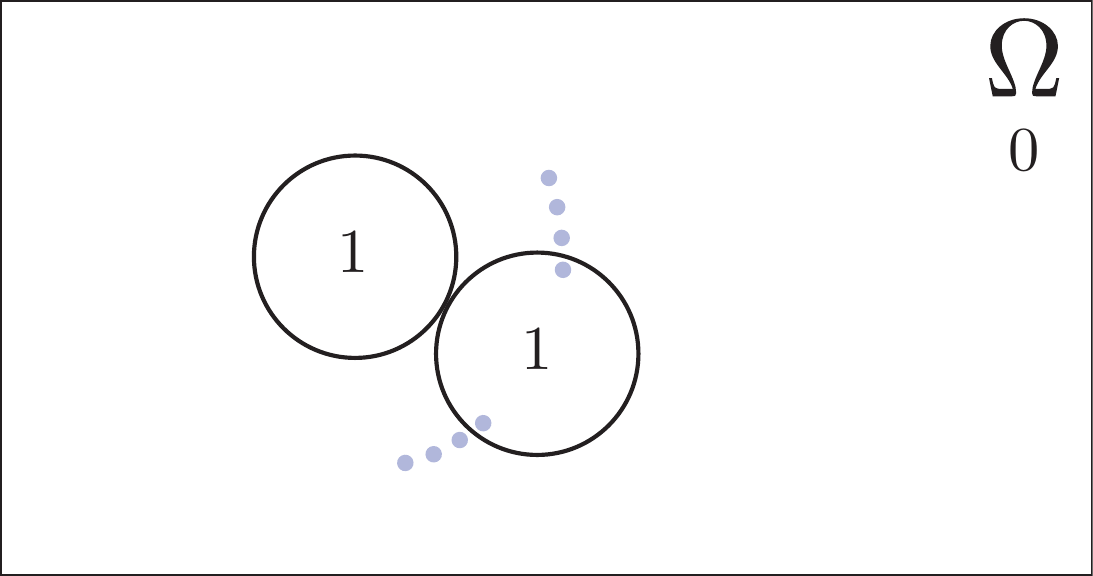}
  \caption{Around this cluster of targets, in our second order approximation with the left ``center'' disk held fixed with $h=1$, there are on average (for the 88 possible positions of the second disk) $c=21.455$ positions where a third disk could be centered such as to be tangent with the center disk, but not cause an error by being discriminated by the second disk. The dots represent a few of these ``good~points.''}
 \label{calculate:c2}
\end{figure}

Next, put down the 4th target. Note that in our second order-error world,  $\mathrm{E}_3$ equals the number of  \ref{sensor_field1}, \ref{sensor_field2}, or \ref{sensor_field3} type situations that are expected to occur.  We take each of these situations represented by $\mathrm{E}_3$ to have on average $c=21.455$ points around them where another disk can be centered such as to not cause an error, while still causing a tangency (see Figure \ref{calculate:c2}). So, as part of calculating the expected number of errors, when calculating the number of tangencies adding a 4th target is expected to create, we subtract off the probability of one of these tangencies being created by the 4th disk being centered at one of the $c\cdot\mathrm{E}_3$ ``good~points.'' So,
\begin{eqnarray*}
\mathrm{E}_4&\approx&\mathrm{E}_3+\frac{3\cdot{C_{6}}}{\hat{A}}-c\left(\frac{\mathrm{E}_3}{\hat{A}}\right) \\
&=&\frac{6\cdot{C_{6}}}{\hat{A}}-\frac{4c\cdot{C_{6}}}{\hat{A}^2}+\frac{c^{2}{C_{6}}}{\hat{A}^3}
\end{eqnarray*}

What we are studying here is the case where a ``center'' disk is held in a fixed position, has a height function of 1, and there is a 2nd disk tangent to this center disk.  A 3rd disk tangent to the center disk can be sufficiently near the first tangency such as to not cause an error -- so that this one grouping of three disks results in only one error. A fourth disk can be brought in tangent to the center disk, but still discriminate itself against the second or third disk. This one grouping still only contains one error.

There can be two tangencies to the first center disk which don't interact, and this creates two separate groupings, each accounting for one error (see Figure \ref{sensor_field5}). 

In summary, note how each tangency creates a hotspot around which tangencies can cluster with only the first one counting as an error. These spots are what we are counting. $\mathrm{E}_n$ is a count of these groupings, as each grouping equals one error, so when we want to predict how many ``good'' tangent spots there will be for the $(n+1)$th target to land, we know it will be approximately proportional to  $\mathrm{E}_n$. Continuing in this way,
\begin{eqnarray*}
\mathrm{E}_5&\approx&\mathrm{E}_4+\frac{4\cdot{C_{6}}}{\hat{A}}-c\left(\frac{\mathrm{E}_4}{\hat{A}}\right) \\
&=&\frac{10\cdot{C_{6}}}{\hat{A}}-\frac{10c\cdot{C_{6}}}{\hat{A}^2}+\frac{5c^{2}{C_{6}}}{\hat{A}^3}-\frac{c^{3}{C_{6}}}{\hat{A}^4}
\end{eqnarray*}
Finally,
\begin{equation*}
\mathrm{E}_n\approx\mathrm{E}_{n-1}+\frac{(n-1){C_{6}}}{\hat{A}}-\frac{c\cdot\mathrm{E}_{n-1}}{\hat{A}}
\end{equation*}
which is \eqref{second_order1}. 
\end{proof}
Continuing on with our proof of Theorem \ref{second_order_theorem}, note that \eqref{second_order1} is a first order linear inhomogeneous recurrence relation with variable coefficient(s). This class of recurrence relations has an elegant solution \cite{recurrencereference}, given by the following lemma.

\begin{lemma}[{\hspace{-.1 mm}\cite{recurrencereference}}]
The solution to a general first order linear inhomogeneous recurrence relation with variable coefficients in the form $a_{n+1} = f_n a_n + g_n \,$, $f_n \neq 0$ is given by 
\begin{align}
\label{recurrence:solution}
a_n & = \left(\prod_{k=0}^{n-1} f_k \right) \left(a_0 + \sum_{m=0}^{n-1}\frac{g_m}{\prod_{k=0}^m f_k}\right) \ \  \text{for}  \ n  \geq 1,
\end{align}
where $a_0$ is the initial term in the sequence $a_0,a_1, a_2, \dots$.
\end{lemma}
\begin{proof}
The starting point for this derivation is $a_{n+1} = f_n a_n + g_n \,$, $f_n \neq 0$. We subtract one term and divide by $\displaystyle \prod_{k=0}^n f_k$,

\begin{align}
\frac{a_{n+1}}{\prod_{k=0}^n f_k} - \frac{f_n a_n}{\prod_{k=0}^n f_k} &= \frac{g_n}{\prod_{k=0}^n f_k} \nonumber \\
\frac{a_{n+1}}{\prod_{k=0}^n f_k} - \frac{a_n}{\prod_{k=0}^{n-1} f_k} &= \frac{g_n}{\prod_{k=0}^n f_k} \nonumber
\end{align}
Let
\[
 A_n=
  \begin{cases}
  \normalsize{ \frac{a_n}{\prod_{k=0}^{n-1} f_k}} & \text{if } n \geq 1\\
   a_0       & \text{if } n=0
  \end{cases}
\]

Then,
\begin{align}
A_{n+1} - A_n &= \frac{g_n}{\prod_{k=0}^n f_k} \nonumber \\
\sum_{m=0}^{n-1}(A_{m+1} - A_m) &= A_n - A_0 = \sum_{m=0}^{n-1}\frac{g_m}{\prod_{k=0}^m f_k} \nonumber \\
\frac{a_n}{\prod_{k=0}^{n-1} f_k} &= A_0 + \sum_{m=0}^{n-1}\frac{g_m}{\prod_{k=0}^m f_k} \nonumber \\
a_n & = \left(\prod_{k=0}^{n-1} f_k \right) \left(A_0 + \sum_{m=0}^{n-1}\frac{g_m}{\prod_{k=0}^m f_k}\right),  \nonumber
\end{align}
which is \eqref{recurrence:solution}, the result we desired.

\end{proof}
We are now ready to prove Theorem \ref{second_order_theorem}.
\begin{proof}
In \eqref{recurrence:solution} take $a_n=\mathrm{E}_n=\mathrm{E}_{n-1}\left(1-\frac{c}{\hat{A}}\right)+\frac{(n-1){C_{6}}}{\hat{A}}$,  then substitute $n \mapsto n + 1$ in this relation to get $a_{n+1}=\mathrm{E}_{n+1}=\mathrm{E}_{n}\left(1-\frac{c}{\hat{A}}\right)+\frac{n{C_{6}}}{\hat{A}}$. Then we compare this equation with $a_{n+1} = f_n a_n + g_n \,$, and take $f_n=\left(1-\frac{c}{\hat{A}}\right)$ and $g_n=\frac{n{C_{6}}}{\hat{A}}$. Finally, noting that $A_0=0$ as $a_0=0$ (if no targets are put down, there are no errors), we plug what we know into \eqref{recurrence:solution} to get,
\begin{eqnarray}
\mathrm{E}_n&\approx&\left(1-\frac{c}{\hat{A}}\right)^n\left(\sum_{m=0}^{n-1}\frac{\frac{m{C_{6}}}{\hat{A}}}{\left(1-\frac{c}{\hat{A}}\right)^{m+1} }\right)\\
&=&\frac{{C_{6}}}{\hat{A}}\cdot\left(1-\frac{c}{\hat{A}}\right)^{n-2}\sum_{m=1}^{n-1}m\left(\frac{1}{1-\frac{c}{\hat{A}}}\right)^{m-1} \label{unsummed}
\end{eqnarray}
This is in the form of $\displaystyle\sum_{k=1}^{n}kr^{k-1}$, a solution to which can be found by differentiating the standard geometric series formula with respect to $r$,
\begin{eqnarray}
\sum_{k=0}^{n}ar^k &= &\frac{a(1-r^{n+1})}{1-r}\nonumber\\ 
\label{sum_formulae}
\frac{d}{dr}\sum_{k=0}^nr^k &= &\sum_{k=1}^n kr^{k-1}=
\frac{1-r^{n+1}}{(1-r)^2}-\frac{(n+1)r^n}{1-r}
\end{eqnarray}
Applying \eqref{sum_formulae} to \eqref{unsummed} gives, after some algebra,

\begin{equation*}
\mathrm{E}_n\approx\frac{{C_{6}}}{c^{2}}\left(\left(1-\frac{c}{\hat{A}}\right)^{n}\cdot \hat{A}-\hat{A}+nc\right)
\end{equation*}

which is a simple solution to our recurrence relation \eqref{second_order1}, and is the desired result~Theorem~\ref{second_order_theorem}.
\end{proof}

\subsection{Asymptotic Behavior}
\label{Asymptotic_Behavior}
Given an understanding of how errors in target enumeration accumulate for small $n$, the question then arises, how does Theorem \ref{theoretical_target_enumeration_theorem} behave for large $n$? In answering this question, a connection arises to a problem in probability called the \emph{coupon collector's problem}. Given a set of $p$ coupons of all different types, how many coupons does the coupon collector expect to need to draw with replacement before having at least one copy of each type of coupon? A similar question is, ``How many coupons does the coupon collector expect to need to draw with replacement before having at least $q$ copies of each type of coupon?'' A solution to this more difficult multiple coupon problem is given by \cite{shepp}. The link between this multiple coupon problem and Euler characteristic integrals is seen if we note that for some integer $s$ the $\{h=s\}$ level set covers all of our sensor field $\Omega$, then by Theorem \ref{theoretical_target_enumeration_theorem} this level set contributes $s$ units to the target enumeration count. If, for instance, each sensor in a sensor field picks up  10 targets, then the $\{h=10\}$ level set will cover all of $\Omega$, and Theorem \ref{theoretical_target_enumeration_theorem} will return a target count of 10.  In \cite{shepp} it is shown that for large $q$, the expected number of coupons the coupon collector needs is asymptotic to $pq$ by the law of large numbers. Recall that our sensor field is a subset $\Omega \subset \mathbb{R}^2 $ and $n$ targets are placed on our sensor field uniformly and discretely. Thus, for a regular grid of sensors  that is $x$ sensors long and $y$ sensors wide, and for targets that are the same shape and size of the sensors and discretely and uniformly  placed over the centers of sensors, this is saying that in the limit as $\ n \to \infty$, every $xy$ targets placed over the sensor field increases by one the number of the highest level set that is expected to cover all of $\Omega$. Then, in the limit as $\ m \to \infty$ if you put down $mxy$ targets,  you would expect the $\{h=m\}$ level set to cover $\Omega$, which would account for $m x  y$ targets, leaving a total of $m x  y-m x  y=0$ targets to fill the higher level sets. In other words, our examination shows that, in the limit as $\ m \to \infty$, if you put down  $mxy$ targets, the highest non-empty level set of $h$ covers $\Omega$, and all higher level sets are empty and do not contribute to the Euler characteristic integral. 

\begin{theorem}[Asymptotic behavior of sensor fields]
\label{asymptotic_thm}
For any sensor field (in a region $\Omega$), and for any (fixed) target shape, let $h'$ be the height function which is the return of the sensor field when each sensor has a target centered over it. Let $H= \int_{\Omega}h'\,d\chi$. Let $\alpha$ be the number of sensors in our sensor field.  Then when $n$ targets are uniformly and discretely put down on our sensor field, where $n=m\cdot \alpha$ for some positive integer m, and where $h$ is the height function coming off this sensor field, we have

\begin{align}
\mathrm{E}\left(\int_{\Omega}h\,d\chi\right)=Hm, \ \  \text{as}  \ n \to \infty
\label{asymptotic_result}
\end{align}
\end{theorem}
\begin{proof}
This follows from  \eqref{linearity_relation} and the above reasoning in this section (Section \ref{Asymptotic_Behavior}). Note that when we put down $m\cdot \alpha$ targets, in the limit as $ \ m \to \infty$ we expect each sensor to have $m$ targets above it. Thus, in the limit $\mathrm{E}(h)=m\cdot h'$. Then by~\eqref{linearity_relation} $\int_{\Omega}m \cdot h'\,d\chi= m\int_{\Omega}h'\,d\chi =Hm.$
\end{proof}

\subsection{Higher order error estimates and some applications}
\label{applications}
To improve our second order  error estimate, we try to force the error formula to be asymptotically correct (see  Theorem \ref{asymptotic_thm}). 

Let $f_c(n)=\left(n-\frac{{C_{6}}}{c^{2}}\left(\left(1-\frac{c}{\hat{A}}\right)^{n}\cdot \hat{A}-\hat{A}+nc\right)\right)$. If we take the limit of  $f_c(n)$ as $n \to \infty$, then

\begin{align*}
\lim_{n\to\infty}\left(n-\frac{{C_{6}}}{c^{2}}\left(\left(1-\frac{c}{\hat{A}}\right)^{n}\cdot \hat{A}-\hat{A}+nc\right)\right) &= \\
 \lim_{n\to\infty} \left(n\cdot \left(1-\frac{{C_{6}}}{c}\right)+\frac{\hat{A}\cdot{C_{6}}}{c^2}\right),
\end{align*}
when $0<c<\hat{A}$. Then, to force the asymptotic result \eqref{asymptotic_result}, we let 
 
 \begin{align}
c=\frac{{C_{6}}}{1-\frac{H}{\alpha}}
\end{align}

such that 
\begin{align*}
\lim_{n\to\infty}\left(n-\frac{{C_{6}}}{c^{2}}\left(\left(1-\frac{c}{\hat{A}}\right)^{n}\cdot \hat{A}-\hat{A}+nc\right)\right) &= \\
Hm+O(1), \ \  \text{as}  \ n \to \infty
\end{align*}

which is in agreement with \eqref{asymptotic_result} up to a constant term. Then, note that $f_c(n)$ increases as $c$ increases, such that $f_c(n)$ increases when $\frac{H}{\alpha}$ increases. Thus, $\frac{H}{\alpha}$ represents how poorly our targets clump together -- the quantity $\frac{H}{\alpha}$ is a measure of the ``clumpiness'' of our targets, or how poorly our targets overlap each other when placed near each other. A large $\frac{H}{\alpha}$ value means targets do not clump together well. Targets that clump together more easily lead our expected Euler characteristic integral to decrease, while targets which clump together less easily would make for larger expected Euler characteristics. For instance, targets which are $1\times 1$ squares would, when placed down densely on a discrete sensor grid, meld together and become impossible to differentiate, such that the quantity $\frac{H}{\alpha}$ would be small. For targets which have more complex geometries, such as discrete disks, we would expect $\frac{H}{\alpha}$ to be larger.

We find that \eqref{final_second_order} with $c=\frac{{C_{6}}}{1-\frac{H}{\alpha}}$ acts a good estimate of the expected value of the Euler characteristic integral. For values of $n$ ranging from 100 to $10\,000$, we simulated, for each $n$, $1\,000$ sensor fields of dimension $500\times 500$, where on each field $n$ radius $6$ targets were  placed down discretely and under a uniform distribution. Then, for each random sensor field we calculated its Euler characteristic integral, and averaged our observed values of the integral for each $n$ (average of $1\,000$ trials). Then, in order to calculate the expected Euler characteristic integral for each $n$, we first calculated $H$ by constructing a  $500\times500$ sensor field with a disk centered at each of the $240\,100$ points sufficiently far from the boundary, and then integrating over this field. This process gives $H=109$. 

See Table \ref{summary:formulas} and Figure \ref{numerical:results:plot}  for a comparison of our expected Euler characteristics to the observed numerical results. Observed data were generated from a virtual sensor field environment created in GNU Octave \cite{eaton:2002}. We find that with $c=\frac{{C_{6}}}{1-\frac{H}{\alpha}}$, our second order error estimate formula  \eqref{final_second_order} behaves excellently for large $n$, excellently for small $n$ (when $f_c(n)$ is insensitive to the value of $c$), and it behaves acceptably for midrange values of $n$.

In order to derive an improved point estimator for $n$, we note that for a fixed target radius and a fixed sensor field area, and with ${C_{6}}\leq c<\hat{A}$, the function defined by taking $n$ to the value given by
\begin{equation}
n-\frac{{C_{6}}}{c^{2}}\left(\left(1-\frac{c}{\hat{A}}\right)^{n}\cdot \hat{A}-\hat{A}+nc\right) \label{injective}
\end{equation}
 is injective.

Arguing by contradiction, we assume there are $n_1$ and $n_2$, $n_1\neq n_2$, and $n_1>n_2$, such that,
\begin{align*}
n_1-\frac{{C_{6}}}{c^{2}}\left(\left(1-\frac{c}{\hat{A}}\right)^{n_1}\cdot \hat{A}-\hat{A}+{n_1}c\right)&=n_2-\frac{{C_{6}}}{c^{2}}\left(\left(1-\frac{c}{\hat{A}}\right)^{n_2}\cdot \hat{A}-\hat{A}+{n_2}c\right)\\
\shortintertext{then}
\underbrace{\frac{(n_1-n_2)}{\hat{A}}\left(\frac{c^2}{{C_{6}}}-c\right)}_{\mathclap{\geq0 \text{ by assumptions}}}&=\underbrace{\left(1-\frac{c}{\hat{A}}\right)^{n_1}-\left(1-\frac{c}{\hat{A}}\right)^{n_2}}_{\mathclap{<0 \text{ by assumptions}}}
\end{align*}
which under our assumptions (${C_{6}}\leq c<\hat{A}$ and $n_1>n_2$) is a contradiction. Therefore, $n_1=n_2$ and the function which takes $n$ to the value given in \eqref{injective} is injective.

Thus, given an observed value $n_{\text{obs}}$ of $\int_{ \mathbb{R}^2}h\,d\chi$, we can solve 
\begin{equation}
\label{improved_estimate}
n_{\text{obs}}=\hat{n}-\frac{{C_{6}}}{c^{2}}\left(\left(1-\frac{c}{\hat{A}}\right)^{\hat{n}}\cdot \hat{A}-\hat{A}+\hat{n}c\right)
\end{equation}
numerically for $\hat{n}$, and obtain a better estimate $\hat{n}$ for $n$ than simply $n_{obs}$. As \eqref{final_second_order} is approximately unbiased, this method of obtaining an improved estimate of $n$ is also approximately unbiased.

Here is one immediate application of the work presented thus far. Imagine a box which has along its bottom a sensor field as described in this paper (made up of a regular grid of sensors). Under the center of each sensor is a magnet. Then, drop a handful of disk-shaped magnets into the box. The magnets under the sensors will attract the disks, and the disks will appear to fall into the box under a discrete uniform distribution (we imagine here that the sides of the box are inelastic, so that disks don't bounce off the sides). Then, the Euler characteristic integral will give us an approximation for the number of disks in the box. We can then shake the box, and arrive at a new uniform distribution of targets, and a new value of the Euler characteristic integral. We can repeat this process many times, calculating a new value of the Euler characteristic integral each time. Then, we can average our Euler characteristic integral values, and plug this average into (\ref{improved_estimate}) as our $n_{\text{obs}}$. This will give us a value $\hat{n}$, which will be a very close estimate for the true number of disk-shaped magnets that are in the box! Note that this process works no matter how many, or how few sensors are at the bottom of the box (many, many, small sensors would work, and so would a few large sensors).

Looking to expand our assumptions, an interesting case is when the targets' centers are continuously placed on our discrete sensor field -- a much more useful application than the discrete placement case. Imagine, for example, a discrete sensor field in front of you. You drop a pile of quarters onto it. Almost surely, not a single quarter will land with its center exactly at a point where four sensors meet; each quarter has the opportunity to drift in one direction or another. The slightest perturbation in a quarter's center coordinate away from an integer will cause it to land on sensors such as to distort its perceived shape and cause new and unforeseen errors (see Figure \ref{continuousplacementfig}). Research needs to be done into this continuous placement case.

\begin{figure}
\centering
      \includegraphics[width=0.2\textwidth]{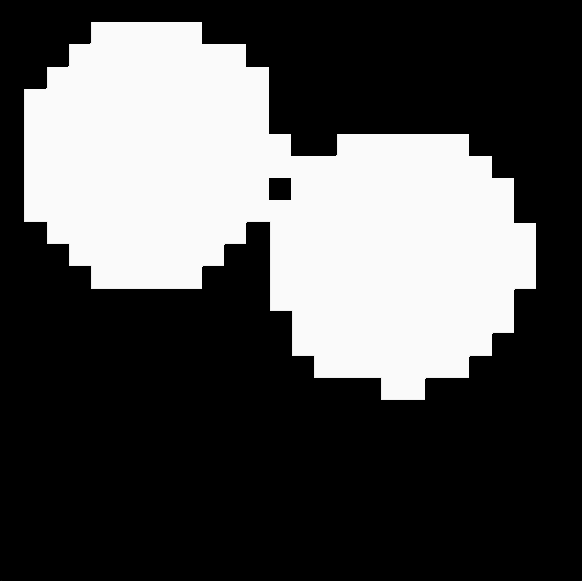}
  \caption{This is the level set generated from two disks of radius 6 falling onto a discrete sensor field with continuous placement of their centers. New behavior arises in this case. Here  $\int_{\Omega}h\,d\chi=0$, which does not occur with discrete placement.}
 \label{continuousplacementfig}
\end{figure}

Looking to further expand past our assumptions, the case when $r>7$ is most interesting. With $r>7$ disks placed discretely on a discrete sensor field, more than simple tangencies can occur. Discretized disks of larger radii can exhibit concave behavior when two identical disks join in certain precise ways to create a single element in the level sets with an Euler characteristic of zero. Larger disks can also exhibit another type of ``locking behavior'' where two disks join together to have an integral with respect to the Euler characteristic of \emph{three} (see Figure \ref{examples:of:discrete:placement:errors}). The number of positions around a disk that another disk's center can be placed to cause these different effects varies from radius to radius, and no pattern seems to emerge (see Table \ref{error:types:positions} and Figure \ref{error:types:positions:plot}). 

An analysis similar to that done for radii less than 8 should be applicable to greater radii.

\begin{table}[tp]
\begin{center}

     \singlespaced
     \vspace{-5mm}
     \scalebox{0.75}{
     \begin{tabular}{| l | c | c | c | c | c | c | r |}
     
    \hline
    
    &\multicolumn{3}{c|}{Error Type}  & &\multicolumn{3}{c|}{Error Type} \\
   \cline{2-4}  \cline{6-8}
  radius&0&1&3&radius&0&1&3 \\
  \hline
1&0&8&0&51&168&584&80\\ 
2&0&24&0&52&152&600&88\\ 
3&0&40&0&53&112&584&96\\ 
4&0&56&0&54&136&576&88\\ 
5&0&72&0&55&120&632&80\\ 
6&0&88&0&56&200&600&96\\ 
7&0&104&0&57&176&632&112\\ 
8&8&112&8&58&248&632&112\\ 
9&0&136&0&59&176&632&88\\ 
10&16&136&16&60&200&696&112\\ 
11&8&160&8&61&184&624&120\\ 
12&8&176&8&62&176&680&112\\ 
13&16&176&16&63&192&672&104\\ 
14&8&208&8&64&248&704&112\\ 
15&24&200&24&65&216&744&88\\ 
16&16&224&16&66&192&688&96\\ 
17&32&208&32&67&184&720&88\\ 
18&40&248&32&68&280&728&152\\ 
19&16&272&16&69&224&712&136\\ 
20&32&256&32&70&208&752&128\\ 
21&16&288&16&71&200&704&112\\ 
22&56&280&48&72&184&736&112\\ 
23&40&312&32&73&280&768&112\\ 
24&32&320&32&74&184&768&104\\ 
25&48&304&48&75&328&752&152\\ 
26&56&344&48&76&208&752&128\\ 
27&72&336&56&77&232&824&120\\ 
28&48&352&48&78&288&792&128\\ 
29&32&376&32&79&216&792&112\\ 
30&72&392&48&80&264&768&128\\ 
31&48&408&40&81&304&856&128\\ 
32&88&416&48&82&248&832&152\\ 
33&48&400&56&83&320&840&152\\ 
34&72&408&48&84&240&848&128\\ 
35&88&424&64&85&376&888&136\\ 
36&56&440&56&86&312&864&152\\ 
37&96&440&56&87&352&832&168\\ 
38&72&472&48&88&264&864&136\\ 
39&136&456&80&89&224&856&152\\ 
40&104&488&64&90&336&888&160\\ 
41&80&488&72&91&336&920&176\\ 
42&96&504&48&92&320&872&160\\ 
43&56&536&56&93&304&944&120\\ 
44&144&512&80&94&376&952&152\\ 
45&104&544&64&95&368&912&184\\ 
46&152&496&88&96&352&960&152\\ 
47&152&552&88&97&344&904&168\\ 
48&96&536&80&98&400&1008&128\\ 
49&136&520&104&99&448&912&200\\ 
50&112&632&72&100&408&952&184 \\
     \hline
    \end{tabular}
    }
    
        \end{center}
\singlespaced
\vspace{-6mm}
  \caption{For varying radius and discrete placement with integer centers, the above table gives the number of positions around a discrete disk where, when another identical disk is placed there,  $\int_{ \mathbb{R}^2}h\,d\chi=$0, 1, or 3 (which we call type 0, 1, or 3 errors).}
  \label{error:types:positions}
\end{table}

\begin{figure}
  \centering
  \subfloat[\tiny{2~targets~of~radius~8, $\int_{\Omega}h\,d\chi=0$}]{\label{radius_8_type_0_error}\includegraphics[width=0.22\textwidth]{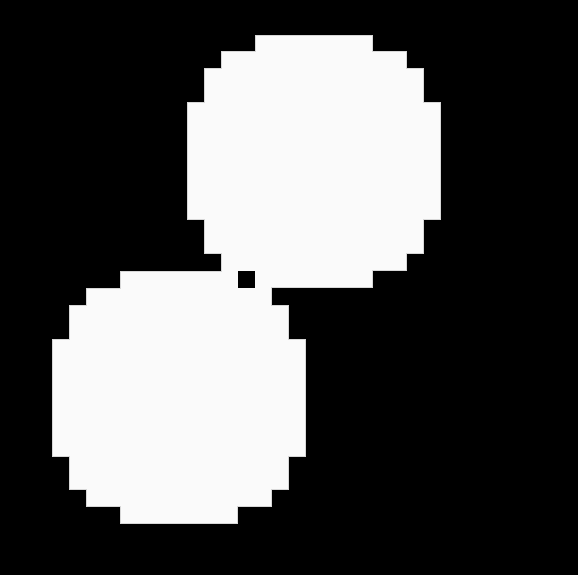}} 
   \hspace{.1cm}              
 \subfloat[\tiny{2~targets~of~radius~8, $\int_{\Omega}h\,d\chi=3$}]{\label{radius_8_type_3_error}\includegraphics[width=0.22\textwidth]{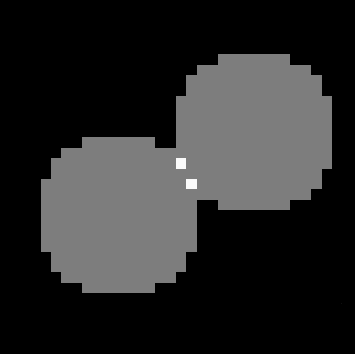}}
 \hspace{.1cm}                   
  \subfloat[\tiny{2~targets~of~radius~10, $\int_{\Omega}h\,d\chi=0$}]{\label{radius_10_type_0_error}\includegraphics[width=0.22\textwidth]{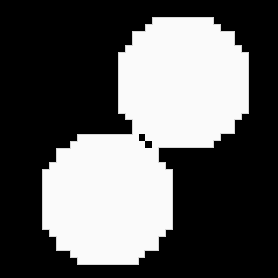}}
  \hspace{.1cm}     
  \subfloat[\tiny{2~targets~of~radius~10, $\int_{\Omega}h\,d\chi=3$}]{\label{radius_10_type_3_error}\includegraphics[width=0.22\textwidth]{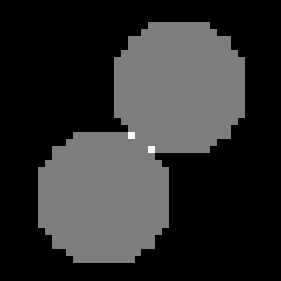}}                
                       
  \caption{Above are examples of the ways targets of radii $>7$ can join together to create different types of errors when evaluating $\int_{\Omega}h\,d\chi$ over a discrete sensor field.}
  \label{examples:of:discrete:placement:errors}
\end{figure}

\begin{figure}[tp]
  \begin{center}
 \hspace*{-20pt}\scalebox{.7}{\input{error_graph.tex}}
 \caption{A plot of table \ref{error:types:positions}.}
 \label{error:types:positions:plot}
  \end{center}
\end{figure}

\begin{table}[tp]
\begin{center}
\singlespaced
     \begin{tabular}{| l | c | p{2.58cm} | p{2.6cm} | p{2.7cm} |}
    \hline
  n&observed&Formula~\eqref{tangency_formula}&Formula~\eqref{final_second_order}, $\textstyle{c=21.455}$&Formula~\eqref{final_second_order}, $\textstyle{c=\frac{{C_{6}}}{1-\frac{H}{\alpha}}}$\vspace{2 mm}\\
  \hline
100&98.294&98.25&98.2&98.3\\
200&193.12&192.95&193.0&193.1\\
300&284.72&284.05&284.2&284.6\\
400&373&371.65&372.0&372.9\\
500&458.5&455.65&456.3&458.2\\
600&541.04&536.15&537.3&540.4\\
700&619.55&613.05&614.8&619.8\\
800&697.11&686.45&689.1&696.4\\
900&771.23&756.25&760.0&770.4\\
1\:000&843.67&822.5&827.7&841.8\\
2\:000&1\:446.2&1\,289.65&1\,330.1&1\,431.8\\
3\:000&1\:881.5&1\,401.4&1\,535.1&1\,845.4\\
4\:000&2\:178.9&1\,157.85&1\,468.1&2\,135.3\\
5\:000&2\:390.0&558.9&1\,152.4&2\,338.7\\
6\:000&2\:524.1&-395.4&609.1&2\,481.3\\
7\:000&2\:611.1&-1\,776.6&-142.1&2\,581.4\\
8\:000&2\:663.2&-3\,370.1&-1\,083.6&2\,651.6\\
9\:000&2\:689.8&-5\,390.5&-2\,199.1&2\,701.0\\
10\:000&2\:703.2&-7\,766.25&-3\,473.7&2\,735.8\\
     \hline
    \end{tabular}
    \end{center}
  \caption{For a $500\times500$ discrete sensor field with $n$ radius 6 disks put down uniformly and discretely, the observed value (averaged over 1\:000 trials for each $n$) of $\int_{ \mathbb{R}^2}h\,d\chi$ is presented, along with $\mathrm{E}\left(\int_{ \mathbb{R}^2}h\,d\chi\right)$ calculated using the first order \eqref{tangency_formula} and then the second order  \eqref{final_second_order} (Theorem \ref{second_order_theorem}) with the second order value $c=21.455$ and with the higher order value $c=\frac{{C_{6}}}{1-\frac{H}{\alpha}}$. For this region, ${C_{6}}=85.322$.}
  \label{summary:formulas}
\end{table}

\begin{figure}[tp]
  \begin{center}
\subfloat[Plot of observed values.]{\label{observed_plot}\includegraphics[width=0.44\textwidth]{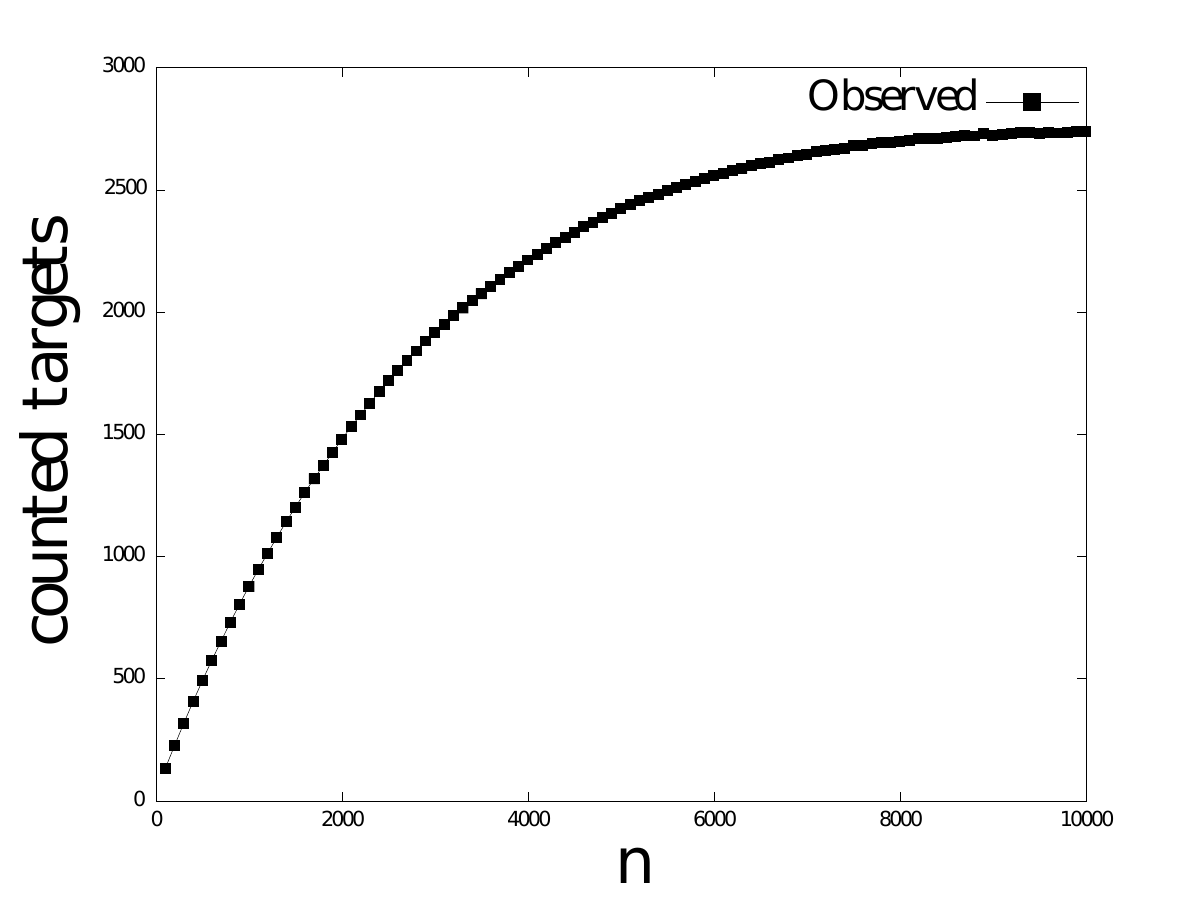}}                
 \subfloat[Plot of our best estimates.]{\label{best_estimate_plot}\includegraphics[width=0.44\textwidth]{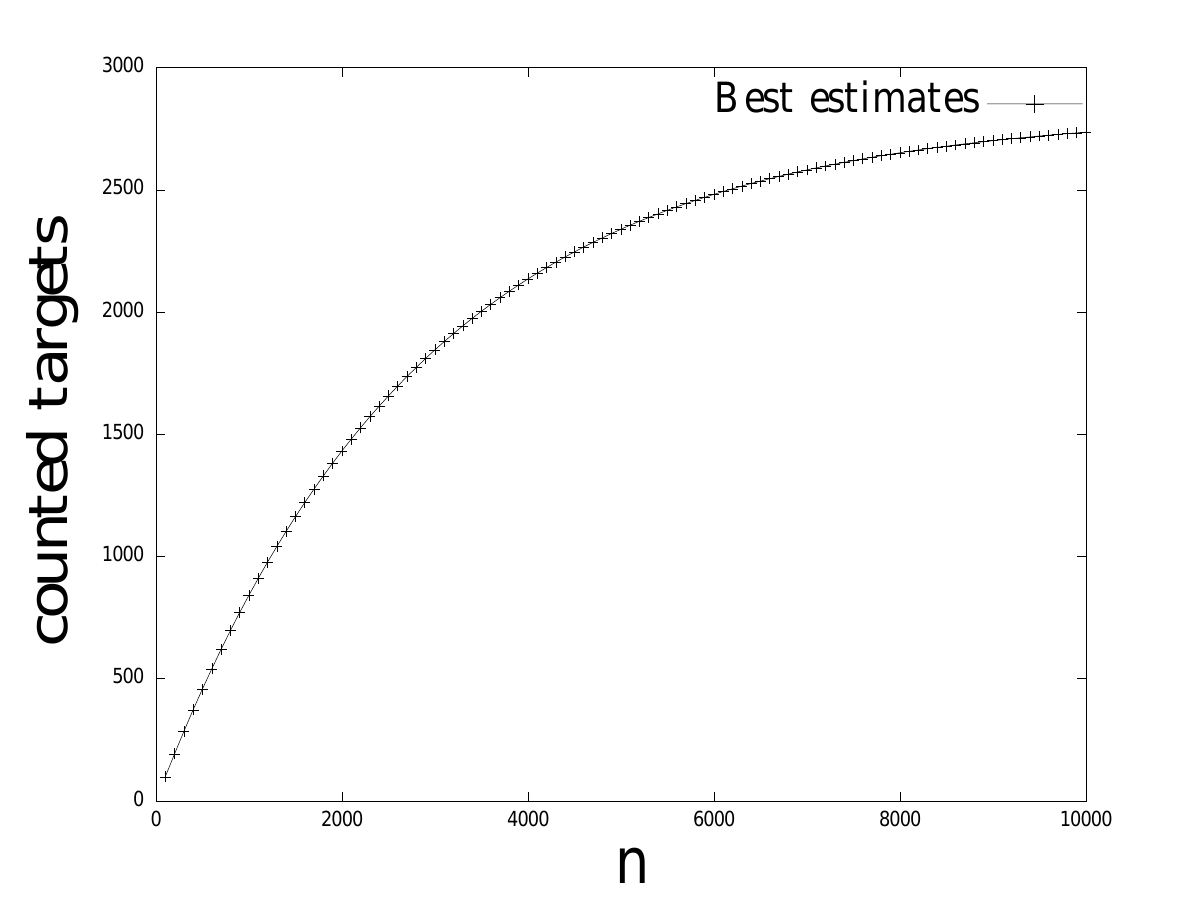}}    \\ 
\hspace*{-10pt} \subfloat[Plot of observed values overlaid our best estimates.]{\label{best_estimate_plot}\includegraphics[width=1.0\textwidth]{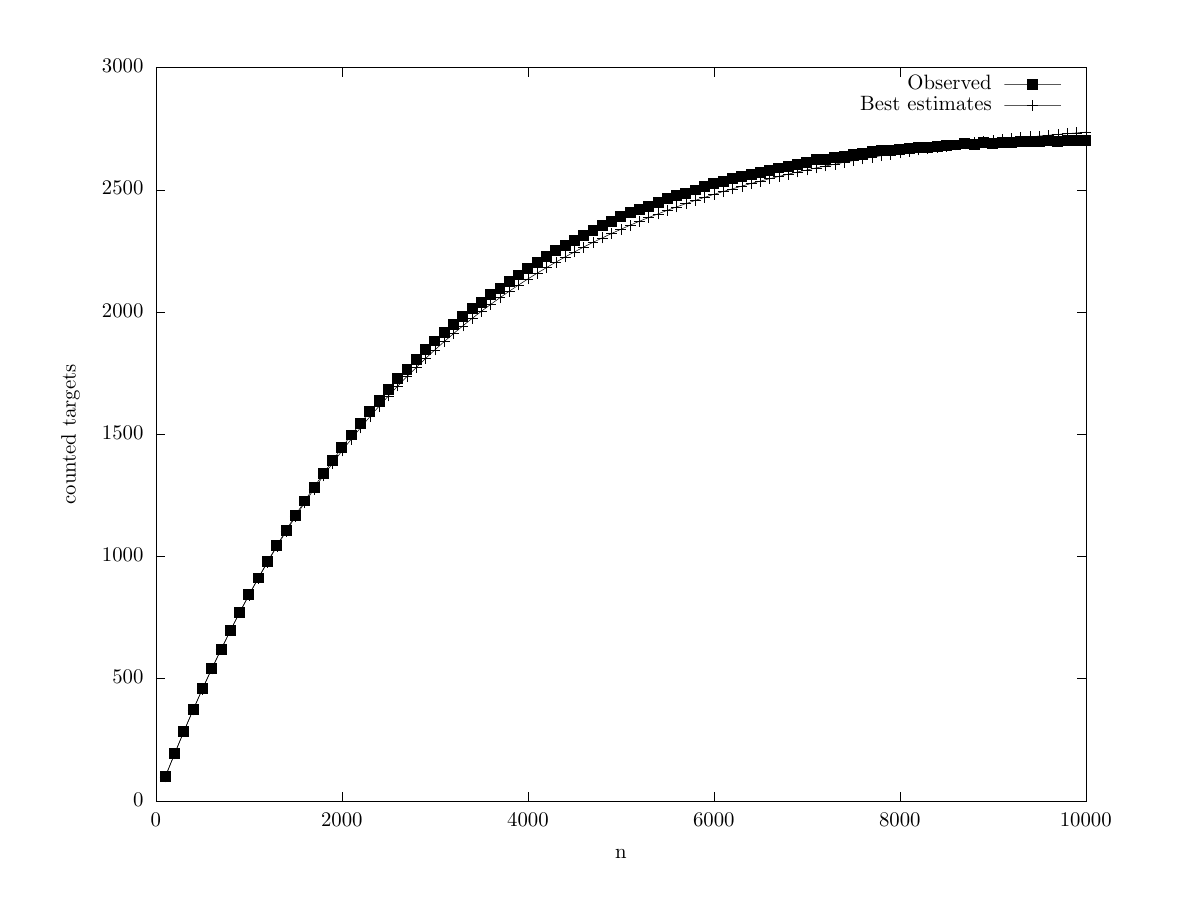}}
 \caption{Plots of Table \ref{summary:formulas}: observed values of the Euler integral alongside our best estimates from \eqref{final_second_order} (Theorem \ref{second_order_theorem}) with $c=\frac{{C_{6}}}{1-\frac{H}{\alpha}}$.}
 \label{numerical:results:plot}
  \end{center}
\end{figure}

\section{Conclusion}
\label{conclusion}
The main result in this paper is Theorem \ref{second_order_theorem}, which gives an approximation to the bias of the Euler characteristic integral as a point estimator for $n$, the number of targets over a discrete sensor field. Numerical results show that Theorem \ref{second_order_theorem} agrees closely with observations.

But this is just a preliminary analysis, and there is still much to be done. There are also many constraining assumptions we made here which are ripe for relaxing. Areas of further research might include looking into,

\begin{itemize}
\item
Disk-shaped targets with radii greater than 7.
\item
Formalizing what we mean when we say  ``first order,'' ``second order,'' or ``higher order'' approximations. 
\item
Finding exact upper and lower bounds on the expected value of the Euler integral.
\item
Disk-shaped targets with radii that are \emph{not} integer multiples of the side length of our sensors.
\item
The discrete Euler integral's complete probability distribution, or even just its variance.
\item
A generalization of the numerical analysis theory to any distribution of targets. 
\item
A generalization to any target shape (and not just 2-d projections of discrete disks!).
\item
Generalizations to multiple target shapes on the same field.
\item
Non-regular sensor grids.
\item
The error tolerance and stability of the discrete integral in the face of noisy data, broken sensors, or other real world problems.  Future work should include a numerical analysis of the theoretical work done on the Euler characteristic integrals of Gaussian random fields  \cite{bobrowski10}.
\item
The techniques for discrete Euler integral analysis and signal processing suggested in \cite{ghrist09,baryshnikov10}
\item
Improving our target enumeration estimate by analysis of time series data for moving targets, similar to \cite{singh07}.
\item
Extending the work done here to targets that are continuously placed onto a sensor field, as opposed to discretely placed with integer center coordinates.
\end{itemize}

Possible applications of any numerical analysis will require additional work. All theory discussed in this article has been based on a priori knowledge of target shape and target distribution. This suggests possible areas of research into relaxing of these  requirements in some way. We should try to determine estimates for these parameters from the data we are given -- a sensor field.

\nocite{*}
\bibliographystyle{plain}
\bibliography{references}

\end{document}